\documentclass[11pt,reqno]{article}

\usepackage[usenames,dvipsnames]{xcolor}
\RequirePackage[OT1]{fontenc}
\usepackage[colorlinks,citecolor=blue,urlcolor=blue,linkcolor=blue,linktocpage=true]{hyperref}
\usepackage{amsmath}
\usepackage{amssymb,amsthm,amsfonts,amsbsy,latexsym,amsxtra}
\usepackage{graphicx}
\usepackage{todonotes}
\usepackage{appendix}
\usepackage{url}
\usepackage{comment}

\usepackage[dvipsnames]{xcolor}              
\usepackage{graphicx}
\usepackage{bbm,bm}
\usepackage{comment}
\usepackage{mathtools}
\usepackage{placeins}
\usepackage[labelfont={sl}]{caption}
\usepackage[labelfont={normalfont}]{subcaption}
\usepackage{thmtools}

\usepackage{amssymb,amsfonts,amsmath,amsthm,amscd,dsfont,mathrsfs}
\usepackage{url}
\usepackage{enumerate}
\usepackage{hyperref}

\usepackage{tikz}
\usetikzlibrary{decorations.pathreplacing,calligraphy}

\makeatletter
\def\namedlabel#1#2{\begingroup
	#2%
	\def\@currentlabel{#2}%
	\phantomsection\label{#1}\endgroup
}
\makeatother

\numberwithin{equation}{section}

\newtheorem{theorem}{Theorem}[section]

\newtheorem{prop}[theorem]{Proposition}
\newtheorem{lemma}[theorem]{Lemma}
\newtheorem{corollary}[theorem]{Corollary}

\theoremstyle{definition}
\newtheorem{definition}{Definition}[section]

\newtheorem{remark}[theorem]{Remark}

\newcommand{\E}{\mathbb{E}}
\newcommand{\R}{\mathbb{R}}

\newcommand{\N}{\mathbb{N}}

\newcommand{\eps}{\varepsilon}
\renewcommand{\P}{\mathbb{P}}

\renewcommand{\emptyset}{\varnothing}
\newcommand{\X}{\mathbb{X}}

\newcommand{\Nb}{\mathbf{N}}

\newcommand{\Var}{\operatorname{Var}}
\newcommand{\dd}{{\mathrm d}}
\newcommand{\Occ}{{\operatorname{Occ}}}

\usepackage{color}

\allowdisplaybreaks[4]

\begin{document}

\title{Chaos and Superconcentration for Poisson Functionals with Applications in Stochastic Geometry}

\author{
Chinmoy Bhattacharjee\thanks{
Department of Mathematics, University of Hamburg,
Bundesstraße 55, 20146 Hamburg, Germany.
\texttt{chinmoy.bhattacharjee@uni-hamburg.de}
}
\and
Rowan O'Clarey\thanks{
Department of Mathematics, University of Hamburg,
Bundesstraße 55, 20146 Hamburg, Germany.
\texttt{rowan.o.clarey@uni-hamburg.de}
}
}

\date{\today}

\maketitle

\begin{abstract}
	We consider square-integrable functionals of Poisson point processes for which the variance upper bound provided by the classical Poincar\'{e} inequality is suboptimal, a phenomenon known as superconcentration. In this paper, we establish a rigorous mathematical equivalence between superconcentration and the chaotic behaviour of the functional, and certain associated random sets, under perturbations driven by the Ornstein-Uhlenbeck semigroup on the Poisson space. Leveraging the Malliavin-Stein method, we develop general variance identities and bounds for Poisson functionals, providing a unified framework to prove superconcentration, particularly for geometric functionals that can be expressed as a sum of local score functions. We apply our results to rigorously establish superconcentration and the chaotic behaviour in some models of stochastic geometry. Specifically, we analyse horizontal box-crossing indicators in certain critical continuum percolations, as well as the number of vertices with small degrees and the number of isolated $\Gamma$-components in random geometric graphs in the dense regime.
\end{abstract}

\section{Introduction and Main Results}\label{sec:intro}

A central objective of probability theory is to understand the fluctuations of observable statistics, that is, functionals of random systems. Although the theory of concentration of measure has undergone substantial development over the past fifty years, in certain specific examples, the classical concentration framework yields bounds on fluctuations that are not of the correct order. This intriguing phenomenon, commonly referred to as \textit{superconcentration}, has been observed in several areas of probability, sometimes under different terminology, such as sub-mean behaviour or sub-linear variance. A systematic investigation of superconcentration was initiated by Chatterjee in \cite{chatterjee2008b} and further developed in the monograph \cite{Chatterjee}, where a rigorous connection was established between superconcentration and chaotic behaviour of the ground state in certain Gaussian disordered systems, a phenomenon known as \textit{chaos}.

The phenomenon of superconcentration has seen a lot of interest recently in various contexts such as surface growth models \cite{Cha23}, spin glass models \cite{CVVH23}, first passage percolation \cite{Dembin24}, Sherrington-Kirkpatrick model \cite{CWW24}, and random matrices \cite{RvH26}, to name a few. On the other hand, the study of chaos and the related concept of noise sensitivity has long been established, originating with the seminal work of Benjamini, Kalai, and Schramm \cite{BKS} on Boolean functions and Bernoulli percolation. Building on this foundation, there has been a recent surge of interest in understanding noise sensitivity and chaos in spatial growth models and continuum spaces. For instance, Ahlberg et al.\ \cite{Ahlberg14} established noise sensitivity for the critical Poisson Boolean percolation (see also \cite{bhattacharjee2024spectrapoissonfunctionalsapplications}), while other recent works have explored the transition to chaos in last passage percolation in $\mathbb{Z}^d$ \cite{Ahlberg2024} and Brownian last passage percolation \cite{GH24}. 

Alongside these developments, stochastic analysis on the Poisson space, particularly the use of Malliavin calculus and the Ornstein-Uhlenbeck semigroup, have revolutionised variance estimations for Poisson functionals \cite{LastPenrose,PeccatiReitzner}.  Recent breakthroughs have provided sharp tools for functional inequalities, such as restricted hypercontractivity \cite{NourdinPeccatiYang} and tight variance lower bounds \cite{SchulteTrapp}.
The primary aim of this work is to bridge these two active domains: we establish a rigorous equivalence between superconcentration and chaos for Poisson functionals, and provide a general framework to study both phenomena for geometric functionals of point processes.

Let $\eta$ be a Poisson process on a measurable space $(\mathbb{X}, \mathcal{X})$ with a $\sigma$-finite intensity measure $\lambda$ and distribution $\mathbb{P}_\eta$. The process $\eta$ (which we interchangeably interpret as a random counting measure) is a random element in the space $\mathbf{N}$ which is the family of $\sigma$-finite counting measures $\mu$ on $\X$ equipped with the smallest $\sigma$-algebra $\mathcal{N}$ such that the maps $\mu \mapsto \mu(A)$ are measurable for all $A \in \mathcal{X}$. For a measurable function $F:\mathbf{N} \to \R$, a locally finite configuration $\mu \in \mathbf{N}$, and $x\in \X$ with $\delta_x$ representing the Dirac measure at $x$, we will denote by $D_xF(\mu)$ the \textit{add-one cost operator}, defined as
\begin{equation}\label{eq:addone}
	D_xF(\mu)=F(\mu+\delta_x)-F(\mu),
\end{equation}
and by $D_x^-F(\mu)$ the analogously defined \textit{remove-one cost operator}, namely,
\begin{align*}
	D_x^-F(\mu)=F(\mu)-F(\mu-\delta_x), \quad x\in\mu.
\end{align*}
We write $F \in L^2_\eta$ whenever $\E [ F^2(\eta)]<\infty$, and denote by $DF$ the map $(x,\mu) \mapsto D_x F(\mu)$. A classical variance bound for square-integrable functionals of Poisson processes is the \textit{Poincar\'e inequality} (see \cite[Remark 1.4]{Wu2000}): for all $F \in L^2_\eta$,
\begin{align}\label{thm:poincare}
        \operatorname{Var} (F(\eta)) \le \E \int_\X (D_x F(\eta))^2 \lambda(\dd x).
\end{align}
While it is well known that the Poincar\'{e} inequality is tight and cannot be improved for a general Poisson functional, for certain special Poisson functional $F$, the Poincaré inequality can be suboptimal. In the following definition, we make this notion precise, and define superconcentration for Poisson functionals in a similar manner as in \cite[Definition 3.1]{Chatterjee}.
\begin{definition}\label{def:supconc}
	Let $(\mathbb{X}, \mathcal{X})$ be a Borel measurable space. For $s>0$, let $\eta_s$ be a Poisson process on $(\mathbb{X}, \mathcal{X})$ with a $\sigma$-finite intensity measure $\lambda_s$. For $\eps_s \to 0$ as $s \to \infty$, we say that a square-integrable functional $F_s \equiv F_s(\eta_s) \in L^2_{\eta_s}$ is \textit{$\eps_s$-superconcentrated} if 
	\begin{equation*}
		\operatorname{Var}(F_s(\eta_s)) \le \eps_s \;\E \int_\X (D_x F_s(\eta_s))^2 \lambda_s(\dd x).
	\end{equation*}
\end{definition}
Thus, superconcentration means that the fluctuation of $F_s$ is asymptotically much smaller than what is guaranteed by the Poincaré inequality.

It has been shown in \cite{Chatterjee} in Gaussian and discrete settings, that the phenomenon of superconcentration for a functional $f$ (on the respective underlying space) is closely related to a chaotic behaviour of certain random sets associated to $f$ with respect to the initial conditions, that is, when the initial configuration is perturbed even slightly, it results in a significant change in the composition of the random set. In this paper, we show that a similar equivalence also holds for Poisson functionals. Let us first define the notion of chaos for a Poisson functionals, in the same spirit as in \cite[Definition 3.3]{Chatterjee}. To do so, we first need to specify a notion of perturbation of the underlying Poisson process $\eta$ with intensity measure $\lambda$. On the Poisson space, a very natural way to perturb the configuration is through a birth-death process, the so-called \textit{Ornstein-Uhlenbeck (OU)} semigroup. One can define the OU semigroup of operators $(P_t)_{t \ge 0}$ acting on measurable functionals $F:\mathbf{N} \to \R$ (henceforth referred to as \textit{Poisson functionals}), given by
\begin{equation}\label{eq:Meh}
	P_t F (\eta) := \E [F (\eta^t) | \eta] = \E [F (\eta_{{e^{-t}}} + \eta'_{{1-e^{-t}}}) | \eta], \quad t \ge 0,
\end{equation}
where $\eta_{{e^{-t}}}$ is a $e^{-t}$-thinning of $\eta$ (that is, the point configuration obtained by keeping each point in $\eta$ independently with probability $e^{-t}$, see e.g.\ \cite[Chapter\ 5]{LastPenrose}) and $\eta'_{{1-e^{-t}}}$ is a Poisson process with intensity $(1-e^{-t}) \lambda$ independent of $(\eta, \eta_{{e^{-t}}})$.

\begin{definition}\label{definition:chaos}
	Let $\eta_s$ be as in \autoref{def:supconc}. For $\eps_s, \delta_s \to 0$ as $s \to \infty$, we say that a Poisson functional $F_s \equiv F_s(\eta_s) \in L^2_{\eta_s}$ is $(\eps_s,\delta_s)$-chaotic with respect to the OU semigroup if for all $t \ge \delta_s>0$,
	$$
	\E \int_\X (D_x F_s(\eta_s) D_x F_s(\eta_s^t)) \lambda_s(\dd x) 
	\le \eps_s \; \E \int_\X (D_x F_s(\eta_s))^2  \lambda_s(\dd x).
	$$
\end{definition}
As observed in \cite{Chatterjee}, chaos for a functional often results in a chaotic behaviour of an associated random set. More precisely, for $\eps_s, \delta_s \to 0$ as $s \to \infty$, we say that a (almost surely finite) random set $A_s \equiv A_s(\eta_s)$ is \textit{$(\eps_s,\delta_s)$-chaotic} with respect to the OU semigroup when for all $t\ge \delta_s>0$,
\begin{equation}\label{def:Chaos}
	e^{-t}\,\E |A_s^0 \cap A_s^t|  \le \eps_s \E |A_s|,
\end{equation}
where we write $A_s^t : = A_s(\eta_s^t)$ for $t \ge 0$. 
This means that in expectation, the composition of the set $A_s^t$ changes almost completely by any finite time $t \ge \delta_s$, with only a smaller order (relative to its size) number of elements in common with the initial set $A_s \equiv A_s^0$. 
	\medskip
	
\noindent\textbf{Notation.} Throughout the paper, we will often simply write $A_s$ for $A_s^0$ for such set-valued processes. For nonnegative functions $f$ and $g$ on $[0,\infty)$, we write $f \asymp g$ if there exist constants $0 < C_1 \le C_2 < \infty$ such that $C_1 \le f/g \le C_2$. To denote one-sided bounds, we use $f = \mathcal{O}(g)$ or $f \lesssim g$ (respectively, $f = \Omega(g)$ or $f \gtrsim g$) to indicate that $f \le Cg$ (respectively, $f \ge Cg$) for some $C \in (0,\infty)$. Finally, $f \ll g$ signifies that $f(x)/g(x) \to 0$ as $x \to \infty$.

\section{Main results}\label{sec:mainresults}

\subsection{Equivalence of Chaos and Superconcentration}\label{sec:chaosresults}
For functionals of an equilibrium measure of a reversible Markov process, \cite[Theorem 3.5]{Chatterjee} showed an equivalence between the notions of superconcentration and chaos. The argument in its proof applies to Poisson functionals as well, yielding the following equivalence result. For completeness, we provide its short proof in Section \ref{sec:chaosproof}. 
\begin{theorem}[Consequence of Theorem 3.5 in \cite{Chatterjee}]\label{thm:eqv}
	Let $\eta_s$ be as in \autoref{def:supconc}. Suppose that a functional $F_s \in L^2_{\eta_s}$ is $\eps_s$-Superconcentrated with $\eps_s \to 0$ as $s \to \infty$. Then for any $\delta_s>0$ with $\delta_s \to 0$ and $\eps_s/\delta_s \to 0$ as $s \to \infty$, we have $F_s$ is $(\eps_s e^{\delta_s}/\delta_s, \delta_s)$-chaotic.
	Conversely, every $(\eps_s,\delta_s)$-chaotic Poisson functional $F_s$ is also $(\eps_s+\delta_s)$-superconcentrated.
\end{theorem}

Note that since $\delta_s \to 0$ as $s \to \infty$, there exists $C \in (1,\infty)$ such that $e^{\delta_s} \in (1,C]$ for all $s>0$. Thus, from the forward implication above, one also obtains that $\eps_s$-superconcentration for $F$ implies that $F$ is $(C\eps_s/\delta_s, \delta_s)$-chaotic.

As mentioned earlier, chaos for Poisson functionals is often exhibited via a chaotic behaviour of certain associated random sets. A crucial step in showing a similar behaviour in the Poisson setting is the following variance identity. For this result, we require an additional assumption that $DF  \in L^2 (\lambda \otimes \P_{\eta})$, that is,
$$
\E \int_\X (D_x F)^2\,  \lambda(\dd x) <\infty.
$$
\begin{theorem} \label{thm:var}
	If $\eta$ is a Poisson point process on a measurable space $(\mathbb{X},\mathcal{X})$ with $\sigma$-finite intensity measure $\lambda$, and $F\in L^2_\eta$ with $DF \in L^2 (\lambda \otimes \P_\eta)$, then
	\begin{align*}
		\Var(F) &= \int_0^\infty\mathbb{E}\bigg[\sum_{x\in\eta\cap\eta^t}D_x^-F(\eta)D_x^-F(\eta^t)\bigg]\dd t.
	\end{align*}
\end{theorem}
While the above abstract identity doesn't let us identify in general a set that is chaotic in the sense of \eqref{def:Chaos}, it gives us all the necessary information about it. In particular, if $D_x F$ is an indicator random variable, then we can find such a random set.
\begin{corollary} \label{cor:var}
	Let $\eta_s$ be as in \autoref{def:supconc}. Assume $F_s(\eta_s) \in L_{\eta_s}^2, DF_s \in L^2(\lambda_s \otimes \P_{\eta_s})$, and $D_xF_s(\mu) D_xF_s(\mu') \in \{0,1\}$ for all $x \in \X$ and $\mu, \mu' \in \mathbf{N}$.  Then
	\begin{align*}
		\Var(F_s)=\int^\infty_0\mathbb{E}|A_s^0\cap A_s^t|\dd t,
	\end{align*}
	where 
	$$
	A_s^t = \{x\in\eta_s^t: D_x^-F_s(\eta_s^t) \neq 0\}, \quad t \ge 0.
	$$
	Moreover, $F_s$ is $(\eps_s,\delta_s)$-chaotic if and only if  the set $A_s$ is $(\eps_s,\delta_s)$-chaotic in the sense of \eqref{def:Chaos}.
\end{corollary}
\begin{remark}\label{rem:GenChaosSupEq}
    Note that the condition $D_xF_s(\mu) D_xF_s(\mu') \in \{0,1\}$ for all $x \in \X$ and $\mu, \mu' \in \mathbf{N}$ in particular implies that $DF \in \{0,\pm 1\}$. Also note that by \autoref{thm:eqv}, \autoref{cor:var} yields that when $D_xF(\mu) D_xF(\mu') \in \{0,1\}$, then $\eps_s$-superconcentration of $F_s$ implies that the set $A_s\equiv A_s^0$ is $(\eps_s,\delta_s)$-chaotic in the sense of \eqref{def:Chaos} for any $\delta_s$, $s>0$ satisfying $\delta_s \to 0$ and $\eps_s/\delta_s \to 0$ as $s\to \infty$. Conversely, it also yields that if $A_s$ is $(\eps_s,\delta_s)$-chaotic, then $F_s$ is $(\eps_s+\delta_s)$-superconcentrated.
Of course, the results also hold true if $D_x F(\mu) D_xF(\mu') \in \{0,-1\}$. More generally, if either $D_x F(\mu) D_xF(\mu') \in \{0\} \cup [a,b]$, or $D_x F(\mu) D_xF(\mu') \in \{0\} \cup [-b,-a]$  for some $0<a\le b<\infty$, then one obtains
$$
	\Var(F)\asymp \int^\infty_0\mathbb{E}|A^0\cap A^t|\dd t,
$$
Also, the equivalence between the two definitions of functional- and set- chaos as in \autoref{cor:var} holds in this case.
\end{remark}

A weaker condition on the difference operator, namely only requiring that $D_x F$ is uniformly bounded away from zero yields just the forward implication in the second assertion of \autoref{cor:var} above.
\begin{corollary}\label{cor:oneway}
	For $\eta_s$ as in \autoref{def:supconc}, assume  $F_s \in L_{\eta_s}^2$ and that there exists a constant $c>0$ such that
	$
	D_x F_s(\mu)D_x F_s(\mu')\in \{0\} \cup [c,\infty)
	$ for all $x \in \X$ and $\mu, \mu' \in \mathbf{N}$
	Then
	\begin{align*}
		\Var(F_s) \ge c^2 \int^\infty_0\mathbb{E}|A_s^0\cap A_s^t|\dd t.
	\end{align*}
	Moreover, if $F_s$ is $(\eps_s,\delta_s)$-chaotic, then the set $A_s\equiv A_s^0$ is $(c^{-2}\eps_s,\delta_s)$-chaotic in the sense of \eqref{def:Chaos}.
\end{corollary}
Again, as in \autoref{rem:GenChaosSupEq}, the result also holds if uniformly, $D_x F(\mu) D_x F(\mu') \in \{0\} \cup (-\infty,-c]$ for some $c>0$. Also, by \autoref{thm:eqv}, in the setting of \autoref{cor:oneway}, superconcentration of $F_s$ implies that the set $A_s$ is chaotic, but the reverse implication cannot be argued if $DF$ is unbounded.

\subsection{General Variance Estimations} \label{Sec:VarEst}
In the previous section, we considered variance estimation, when $D_xF(\mu) D_xF(\mu')$ is either nonpositive or nonnegative. But in general, one typically has the situation that the product is real-valued. Throughout this section, $\eta$ is a Poisson point process on a measurable space $(\mathbb{X},\mathcal{X})$ with $\sigma$-finite intensity measure $\lambda$. Assume that $F \in L^2_\eta$ and $DF \in L^2 (\lambda \otimes \P_{\eta})$. Then one has the following variance representation by \cite[Theorem 20.2]{LastPenrose} (see also \eqref{eq:Mehler}):
\begin{align*}
	\operatorname{Var}(F(\eta)) &= \int^\infty_0 e^{-t} \int_\X \mathbb{E}[D_xF(\eta) D_x F(\eta^t)]\, \lambda(\dd x)\dd t.
\end{align*}
Consider now a decomposition $D_xF(\eta) = g_1(x,\eta) - g_2(x,\eta)$, with $g_1, g_2 \ge 0$ so that their supports are disjoint. Note that such a decomposition always exists, for instance, one can simply take $g_1= (DF)_+ = DF \mathds{1}\{DF \ge 0\}$ and $g_2 = (DF)_- = - DF \mathds{1}\{DF \le 0\}$. Then by the exchangeability of $\eta$ and $\eta^t$ for $t\geq 0$, one has
\begin{align}\label{eq:vardecomp}
	\Var(F) = &\int^\infty_0 e^{-t} \int_\X \mathbb{E}[g_1(x,\eta) g_1(x,\eta^t)]\,\lambda(\dd x)\dd t \nonumber\\
	&+ \int^\infty_0 e^{-t} \int_\X \mathbb{E}[g_2(x,\eta) g_2(x,\eta^t)]\,\lambda(\dd x)\dd t \nonumber\\
	& \quad-2 \int^\infty_0 e^{-t} \int_\X \mathbb{E}[g_1(x,\eta) g_2(x,\eta^t)]\,\lambda(\dd x)\dd t =: T_1 + T_2 - 2 T_3.
\end{align}

One approach to estimate the variance  of $F$ is then by individually estimating $T_1, T_2$ and $T_3$. Because of the non-negativity of $g_1, g_2$, this allows one to separate the positive and negative contributions to the variance and use the results from Section \ref{sec:chaosresults}, paving the way for a tight variance bound. Moreover, choosing an appropriate decomposition, that is the functions $g_1,g_2$, can greatly assist in estimating the terms $T_1, T_2$ and $T_3$. For instance, one typical example is when $F(\eta) = \sum_{y \in \eta} f(y,\eta)$ is a sum of \textit{score functions}. If one has $f \le 0, Df \ge 0$ or $f \ge 0, Df \le 0$ (see Section \ref{sec:isovert} for an example), then a natural way to decompose $DF$ is 
\begin{equation}\label{eq:Dxdecomp}
    D_x F(\eta) =  \sum_{y \in \eta \cup \{x\}} f(y, \eta+\delta_x) - \sum_{y \in \eta} f(y, \eta) = f(x, \eta+\delta_x) - \Big(\sum_{y \in \eta} - D_x f(y, \eta) \Big).
\end{equation}

In the following result, we first provide a general upper bound for terms such as $T_1$ and $T_2$. Below, $\|\cdot\|_{L_\eta^p}$, $p >0$ denotes the $L^p$-norm for Poisson functionals.

\begin{theorem}\label{thm:L1-L2}
	For a measurable $g: \X \times \mathbf{N} \to [0,\infty)$ with $g(x,\cdot) \in L_\eta^2$ for $\lambda$-almost every $x \in \X$, 
	$$
	\int^\infty_0 e^{-t} \int_\X \mathbb{E}[g(x,\eta) g(x,\eta^t)]\, \lambda(\dd x)\dd t \le \int_\X \|g(x,\eta)\|_{L^2_\eta}^{2}\, \lambda(\dd x).
	$$
	
	If in addition $D_y g(x,\eta) \le 0$ for $\lambda$-almost every $x,y \in \X$, then
	\begin{align*}
		\int^\infty_0 e^{-t} \int_\X \mathbb{E}[g(x,\eta) g(x,\eta^t)]\,\lambda(\dd x)\dd t &\le  2 \int_\X \frac{\|g(x,\eta)\|_{L_\eta^2}^{2}}{1+(1/2) \log (\|g(x,\eta)\|_{L_\eta^2}/\|g(x,\eta)\|_{L_\eta^1})}  \lambda(\dd x).
	\end{align*}
\end{theorem}

The first assertion can be interpreted as the Poincar\'{e} inequality while the second assertion is a version of Talagrand's $L^1 - L^2$ inequality \cite{Cordero-ErausquinLedoux,Talagrand} on the Poisson space in the same spirit of \cite[Theorem 1.6]{NourdinPeccatiYang}, except that we don't require $g$ to be obtained via a difference operation, and hence can apply it to any decomposition as above. 

Similar to \autoref{cor:var}, one typical case we can consider is when $F$ is such that $ |DF| \in \{0\} \cup [a,b]$ for $0<a\le b<\infty$ almost surely. Denote the events
\begin{equation}\label{eq:AxBx}
	A_x:= \{(x,\eta): D_x F(\eta)>0\}\quad \text{and} \quad B_x:= \{(x,\eta): D_x F(\eta)<0\}.
\end{equation}
Then $D_xF(\eta) = (D_xF)_+ - (D_x F)_- =: g_1(x,\eta) - g_2(x,\eta)$, where $g_1 \asymp \mathds{1}(A_x)$ and $g_2 \asymp \mathds{1}(B_x)$ have disjoint supports. With this decomposition of $DF$, the terms $T_1, T_2, T_3$ in \eqref{eq:vardecomp} can be expressed as
\begin{align*}
&T_1  \asymp \int^\infty_0 \mathbb{E} |A^0\cap A^t|\dd t, \qquad T_2  \asymp \int^\infty_0 \mathbb{E} |B^0\cap B^t|\dd t\\
& \text{and} \quad T_3  \asymp \int^\infty_0 \mathbb{E} |A^0\cap B^t|\dd t = \int^\infty_0 \mathbb{E} |B^0\cap A^t|\dd t,
\end{align*}
where $A^t$ (respectively $B^t$), for $t \ge 0$, denotes the set of points $x \in\eta^t$ for which the event $A_x$ (respectively $B_x$) occurs, that is, 
\begin{equation}\label{eq:At}
	A^t = \{x \in \eta^t : D_x^- F(\eta^t)>0\} \quad \text{and} \quad B^t = \{x \in \eta^t : D_x^- F(\eta^t)<0\}.
\end{equation} 
In general, one can obtain such representations of the terms $T_1, T_2, T_3$ for any decomposition $D_x F = g_1 - g_2$ with $g_1, g_2$ supported on disjoint sets. 
In this setting, where one considers a function $g$ which is an indicator function up to a multiplicative constant bounded away from zero and infinity, \autoref{thm:L1-L2} yields the following corollary.
\begin{corollary}\label{cor:bddD1}
	Let $g(x,\eta) = \mathds{1}((x,\eta+\delta_x) \text{ satisfy $P$})$ for some property $P$. Further, let $P^t$, $t \ge 0$ denote the set of points $x \in\eta^t$ such that $(x,\eta^t)$ satisfy $P$. Then
	$$
	\int^\infty_0 \mathbb{E} |P^0\cap P^t|\dd t \le \int_\X \P((x,\eta+\delta_x) \text{ satisfy $P$}) \lambda(\dd x).
	$$
	If in addition $D_y g(x,\eta) \le 0$ for $\lambda$-almost every $x,y \in \X$, then
	\begin{align*}
		\int^\infty_0 \mathbb{E} |P^0\cap P^t|\dd t &\le 2 \int_\X \frac{\P((x,\eta+\delta_x) \text{ satisfy $P$})}{1- (1/4) \log (\P((x,\eta+\delta_x) \text{ satisfy $P$}))}  \lambda(\dd x).
	\end{align*}
\end{corollary}

\medskip

Since we want a tight estimate of the variance of $F$, we also need to find appropriate lower bounds for the terms $T_1$ and $T_2$. To this end, the following result provides a general lower bound similar to \cite[Theorem 1.1]{SchulteTrapp}.

\begin{theorem}\label{thm:Varlb}
	For a measurable $g: \X \times \mathbf{N} \to [0,\infty)$ with $g(x,\cdot) \in L_\eta^2$ for $\lambda$-almost every $x \in \X$, assume that
	$$
	\mathbb{E} \left[\int_{\X^2} \left(D_y g(x,\eta) \right)^2 \lambda^2(\dd x,\dd y)\right] \le \alpha \mathbb{E} \left[\int_\X g(x,\eta)^2 \lambda(\dd x)\right] <\infty
	$$
	for some $\alpha > 0$. Then
	$$
	\int^\infty_0 e^{-t} \int_\X\mathbb{E}[g(x,\eta) g(x,\eta^t)]\,\lambda(\dd x)\dd t \ge \frac{1}{\alpha+1}\mathbb{E} \left[\int_\X g(x,\eta)^2 \,\lambda(\dd x)\right].
	$$
\end{theorem}
Specialising to the case in  \autoref{cor:bddD1} with $g$ as an indicator, we immediately obtain the following corollary.
\begin{corollary}\label{cor:bddD2}
	Let $g(x,\eta) = \mathds{1}((x,\eta+\delta_x) \text{ satisfy $P$})$ for some property $P$. Assume that
	$$
	\mathbb{E} \left[\int_{\X^2} D_y g(x,\eta)^2 \lambda^2(\dd x,\dd y)\right] \le \alpha \int_\X \P((x,\eta+\delta_x) \text{ satisfy $P$}) \, \lambda(\dd x)
	$$
	for some $\alpha > 0$. Then for $P^t$, $t \ge 0$ as in \autoref{cor:bddD1},
	$$
	\int^\infty_0 \mathbb{E} |P^0\cap P^t|\dd t \ge \frac{1}{\alpha+1} \mathbb{E} |P^0|.
	$$
\end{corollary}

While the above results aid in estimating the terms $T_1$ and $T_2$ in the decomposition \eqref{eq:vardecomp}, one still needs to estimate $T_3$ to obtain a tight variance estimate. In the examples of Poisson functionals analysed in this work, either $T_1$ or $T_2$ typically dominates the variance, rendering $T_3$ negligible. Accordingly, up to a relabelling of $g_1$ and $g_2$, we may assume without loss of generality that in \eqref{eq:vardecomp},
\begin{equation}\label{eq:varianceorder}
T_1 =\int^\infty_0 e^{-t} \int_\X \mathbb{E}[g_1(x,\eta) g_1(x,\eta^t)]\, \lambda(\dd x)\dd t \gg \int^\infty_0 e^{-t} \int_\X \mathbb{E}[g_2(x,\eta) g_2(x,\eta^t)]\, \lambda(\dd x)\dd t =T_2.
\end{equation}
It will often be convenient to use Theorems \ref{thm:L1-L2} and \ref{thm:Varlb} to verify the condition \eqref{eq:varianceorder}. Assuming \eqref{eq:varianceorder}, the following lemma and its corollary show that $T_3 \ll T_1$, so that
\begin{align*}
	\Var(F)\asymp \int^\infty_0 e^{-t} \int_\X \mathbb{E}[g_1(x,\eta) g_1(x,\eta^t)]\, \lambda(\dd x)\dd t.
\end{align*}

\begin{lemma}\label{lem:T_3order}
	Let $F \in \operatorname{dom} D$ and $g_1, g_2$ be as in \eqref{eq:vardecomp}. Assume that \eqref{eq:varianceorder} holds true. Then
	\begin{align*}
		\int^\infty_0 e^{-t} \int_\X \mathbb{E}[g_1(x,\eta) g_2(x,\eta^t)]\, \lambda(\dd x)\dd t &= \int^\infty_0 e^{-t} \int_\X \mathbb{E}[g_2(x,\eta) g_1(x,\eta^t)]\, \lambda(\dd x)\dd t \\&\ll \int^\infty_0 e^{-t} \int_\X \mathbb{E}[g_1(x,\eta) g_1(x,\eta^t)]\, \lambda(\dd x)\dd t.
	\end{align*}
\end{lemma}

\begin{corollary}\label{cor:T3}
	Let $A^t, B^t$ be as in \eqref{eq:At}. Assume that \begin{equation*}
		\int^\infty_0 \mathbb{E} |A^0\cap A^t|\, \dd t \gg \int^\infty_0 \mathbb{E} |B^0\cap B^t|\, \dd t .
	\end{equation*} 
	Then
	$$
	\int^\infty_0 \mathbb{E} |A^0\cap B^t|\, \dd t  = \int^\infty_0 \mathbb{E} |B^0\cap A^t|dt  \ll \int^\infty_0 \mathbb{E} |A^0\cap A^t|\, \dd t.
	$$
\end{corollary}

\subsection{Superconcentration for sums of scores}\label{sec:SCres} Given the equivalence of superconcentration and chaos of Poisson functionals, as stipulated in \autoref{thm:eqv}, it is natural to ask how one can prove either superconcentration or chaos for general square integrable Poisson functionals. This question seems to be very challenging to answer in general. In Section \ref{sec:app}, we consider some applications, where by doing a direct variance analysis, we prove superconcentration of certain Poisson functionals, and hence their chaos.

The variance analysis can be simplified in the setting when one has a Poisson functional that can be expressed as a sum of \textit{score functions}. Let $\eta_s$ be a Poisson process on $\X$ with intensity measure $\lambda_s$ as in \autoref{def:supconc}. For measurable score functions $f_s: \X \times \Nb \to \R$, $s>0$, assume that $F_s\equiv F_s(\eta_s)$ has the representation
\begin{equation}
	\label{eq:hs}
	F_s(\eta_s):= \sum_{x \in \eta_s} f_s(x,\eta_s),
\end{equation}
when the sum converges. Here, we implicitly assume via this representation that the functional $F_s$ can be
decomposed as a sum of local contributions, known as scores, from each point $x \in \eta_s$. 
We further assume that the scores satisfy the following:
\begin{itemize}
\item[\namedlabel{A1}{(\textbf{A1})}] There exists a constant $C \in (0,\infty)$ such that $T_s \le C\, \int_\X \E[f_s(x,\eta+\delta_x)^2]\, \lambda_s(\dd x)$, where for $x,y \in \X$, we denote 
    \begin{align*}
        T_s :=&\int_{\X^2} \E[f_s(x,\eta_s+\delta_{x}+\delta_{y}) f_s(y,\eta+\delta_{x}+\delta_{y})] \, \lambda_s^2(\dd x,\dd y)\\
        &\qquad \qquad -\int_{\X^2}\E[f_s(x,\eta+\delta_{x})]\E[f_s(y,\eta+\delta_{y})]\, \lambda_s^2(\dd x,\dd y).
    \end{align*}
    
    \end{itemize}
\begin{itemize}
	\item[\namedlabel{A2}{(\textbf{A2})}] There exists $\eps_s \to 0$ as $s \to \infty$ such that for $\lambda_s$-almost every $y \in \X$,
	$$
	\E[f_s(y,\eta_s+\delta_y)^2]\le \eps_s \,  \E\left[\Big(\sum\limits_{x\in\eta_s}D_yf_s(x,\eta_s)\Big)^2\right].
	$$
\end{itemize}

The following proposition is a direct consequence of an expansion of the variance of $F_s$ given by \eqref{eq:hs} via the Mecke formula (see \eqref{eq:Mecke} below).
\begin{prop}[Superconcentration for sums of scores]\label{thm:supcon_sos}

    Let $F_s \in L^2_{\eta_s}$ be as in \eqref{eq:hs} satisfying the assumption \ref{A1}. Then
    $$
    \Var(F_s)= \mathcal{O}\left(\int\E[f_s(y,\eta_s+\delta_y)^2]\, \lambda_s(\dd x)\right).
    $$
    Moreover, when \ref{A2} holds, then $F_s$ is $(1+C)\eps_s/(1-\sqrt{\eps_s})^2$-superconcentrated for $C$ as in \ref{A1}.
\end{prop}

To demonstrate the broad applicability of our general framework, Section \ref{sec:app} rigorously establishes superconcentration and chaos in three distinct models from stochastic geometry. First, we analyse the $\pm 1$-indicator for left-to-right occupied crossings of the box $W_s= [-s,s]^d$, $s>1$, in critical Boolean and Voronoi percolation. By proving the superconcentration of this functional, we deduce the chaotic behaviour of its associated pivotal points (see Section \ref{sec:perc}). The remaining two applications focus on random geometric graphs generated by a homogeneous Poisson point process on the torus $[0,1]^d/\sim$ in the dense regime, where the average degree $sr_s^d$ diverges as $s \to \infty$. Specifically, we establish superconcentration, and identify a corresponding chaotic set, for the number of vertices with degree strictly less than a fixed integer $k \in \N$, which includes isolated vertices as a natural special case. Finally, we prove superconcentration for the number of $\Gamma$-components in such a random geometric graph, providing sharp variance control for the count of isolated isomorphic copies of a given feasible connected graph $\Gamma$ on $k$ vertices.

The rest of the paper is organised as follows. In section \ref{sec:app}, we consider the three examples mentioned above.  Section \ref{sec:proofs} contains the proofs of all our main results in Section \ref{sec:mainresults}.

\section{Applications}\label{sec:app}
\subsection{Horizontal box-crossings in continuum percolations}\label{sec:perc} The first example we consider is in the context of continuum percolations in $\R^2$ such as Voronoi or Boolean percolations.
Consider a stationary Poisson point process $\eta$ with intensity $\lambda>0$ on $\R^2$. In the Boolean model, the  {\it occupied region} is defined as $\Occ(\eta) = \cup_{x \in \eta}B(x,1)$, where $B(x,1)$ denotes the closed ball of unit radius around $x \in \R^2$. For Voronoi percolation, fix $\lambda=1$ and, for each $x \in \eta$, independently colour the corresponding Voronoi cell in black with probability $p$, or in white, with probability $1-p$. The occupied region $\Occ(\eta)$ is then defined as the union of the black cells. It is well known that there exists a critical parameter $\lambda_c>0$ (respectively $p_c = \frac{1}{2}$ for Voronoi percolation), such that the occupied region in the Boolean model (respectively Voronoi percolation) contains an unbounded component almost surely whenever $\lambda > \lambda_c$ (respectively $p > p_c$), and does not contain any infinite component almost surely otherwise (see e.g.\ \cite{hall85} and \cite{BO06}).

A critical continuum percolation model, specifically its macroscopic properties such as percolation, is typically sensitive to small perturbations  of $\eta$. One commonly considered observable for this purpose is the event that there is a left-to-right (L-R) crossing of the box $W_s: =  [-s,s]^2$, $s>1$, through its occupied region $\Occ(\eta)$. Let $F_s$ denote the $\pm 1$-indicator that there is such an L-R occupied crossing of $W_s$ in $\Occ(\eta)$, that is,
$$
F_s(\eta) = \begin{cases}
	1, &\text{if there is an L-R occupied crossing of } W_s,\\
	-1, &\text{otherwise.}
\end{cases}
$$ 
Further, we call a point $x \in \eta$ as \textit{pivotal} for $F_s$ if the removal of $x$ from $\eta$ would flip the value of $F_s$. In particular,we define the {\it pivotal set} associated with $F_s$ and $\eta$ as
$$
\mathcal{P}_{\eta}^s := \{x\in \eta : D^-_x F_s(\eta) \neq 0\}.
$$
Also, denote by $\alpha_4(s)$, $s > 1$ the so-called \textit{$4$-arm probability}, that is, the probability that there are $4$ arms of alternating types (occupied/vacant for Boolean, and black/white for Voronoi) from $\partial W_1$ to $\partial W_s$ inside the annulus $W_s \setminus W_1$ (see \cite{bhattacharjee2024spectrapoissonfunctionalsapplications} for further details). A crucial property of these probabilities, established in \cite[Equation (20)]{MS22} for the critical Boolean model and \cite[Proposition 1.15]{Van} for critical Voronoi percolation, is that there exists a universal constant $c \in (0,2)$ such that for any $1<s<\infty$,
\begin{equation*}
\alpha_4(s)\ge c(1/s)^{2-c}.
\end{equation*}
Note as a consequence that $s^2 \alpha_4(s) \ge c s^c \to \infty$ as $s \to \infty$.

\begin{theorem}\label{thm:crossing}
	The crossing functional $F_s$, $s > 1$ is $\left(C \,(s^2 \alpha_4(s))^{-1} \right)$-superconcentrated for some constant $C \in (0,\infty)$, that is,
	$$
	\Var(F_s) \le \frac{C}{s^2\alpha_4(s)} \int_{\R^2} \E (D_x F_s(\eta))^2 \, \dd x.
	$$
	Moreover, 
	\begin{equation*}
		\Var(F_s)= 4\int_0^\infty\E|\mathcal{P}_{\eta}^s \cap\mathcal{P}_{\eta^t}^s|\, \dd t.
	\end{equation*}
	Consequently, there exists a constant $C'\in (0,\infty)$ such that the set of Pivotal points $\mathcal{P}_{\eta}^s, s >1$ is $(C' \,(\delta_s \, s^2 \alpha_4(s))^{-1}, \delta_s)$-chaotic for any $\delta_s>0$ with $\delta_s \to 0$ and $\delta_s \, s^2 \alpha_4(s) \to \infty$ as $s \to \infty$.
\end{theorem}
\begin{proof}
	Since $F_s \in \{\pm 1\}$, we have that $\Var(F_s(\eta))  \le 1$ for both critical Boolean and Voronoi percolation. On the other hand, consider the r.h.s.\ of the Poincar\'{e} inequality \eqref{thm:poincare}, that is the integral $\int_{\R^2} \E (D_x F_s(\eta))^2 \, \dd x$. Since $(D_x F_s(\eta))^2 \in \{0,4\}$ in both models, we have by  \cite[Proof of lower bound in (H4), Page 46]{bhattacharjee2024spectrapoissonfunctionalsapplications} that there exists some constant $C_1 \in (0,\infty)$ such that
	$$
	\int_{\R^2} \E (D_x F_s(\eta))^2 \, \dd x \ge 4 \int_{W_{s/2}} \mathbb{P}(D_x F_s \neq 0) \, \dd x \ge C_1 s^2 \alpha_4(s)
	$$
	for the Boolean model while for the Voronoi model, the same follows by \cite[Proposition 3.15]{bhattacharjee2024spectrapoissonfunctionalsapplications}, that is,
	$$
	\int_{\X} \E (D_{(x,a)} F_s(\eta))^2 \, \dd x \, \nu(\dd a) \ge C_1 s^2 \alpha_4(s),
	$$
	where $\X = \R^2 \times \{0,1\}$ with $0,1$ representing the colours black and white, and $\nu = \frac{1}{2} \delta_{0} + \frac{1}{2} \delta_1$. This together with the fact that $\Var(F_s(\eta))  \le 1$ yields the superconcentration.
	\medskip
	
	For the second assertion, first consider the critical Boolean percolation. In this case, $D^-_x F_s(\eta) \in \{0,2\}$, and hence $D^-_x F_s(\eta) D^-_x F_s(\eta^t) \in \{0,4\}$ for all $x \in \eta$. 
	
	\noindent On the other hand, for critical Voronoi percolation, one has $D^-_x F_s(\eta) \in \{-2,0,2\}$ for all $x \in \eta$. But note that $D^-_x F_s(\eta) D^-_x f_s(\eta^t) \ge 0$ for all $t \ge 0$ and $x \in \eta \cap \eta^t$. Indeed, if $D_x^- F_s(\eta) = -2$, that is, removing $x$ creates an L-R crossing of $W_s$ (with none being there before), then the cell associated to $x$ must be white, and hence removing it from $\eta^t$ cannot destroy any existing black crossing of $W_s$ and thus $D^-_x F_s(\eta^t) \in \{-2,0\}$. A similar argument applies when $D_x^- F_s(\eta) = 2$, showing $D^-_x F_s(\eta) D^-_x F_s(\eta^t) \in \{0,4\}$. 
	
	Applying now \autoref{cor:var} for $F_s' = F_s/2$, and noting that the set $A_s^t$ in \autoref{cor:var} is exactly $\mathcal{P}_{\eta^t}^s$ with $\eta^t$ as in \eqref{eq:Meh} for these percolation models, we obtain the second assertion, that is,
	\begin{equation*}
		\Var(F_s)= 4\int_0^\infty\E|\mathcal{P}_{\eta}^s \cap\mathcal{P}_{\eta^t}^s|\, \dd t.
	\end{equation*}
	
	Finally, since by the first assertion we have that $F_s$ is $\left(C \,(s^2 \alpha_4(s))^{-1} \right)$-superconcentrated, by \autoref{rem:GenChaosSupEq}, combining \autoref{thm:eqv} and \autoref{cor:var} yields the final assertion.
\end{proof}

\subsection{Vertices with small degrees in random geometric graphs}\label{sec:isovert} Consider the $d$-dimensional torus $\mathbb{T}=[0,1]^d/\sim$, $d \in \N$, and let $\eta_s$ be a stationary Poisson point process with intensity $s>0$ on $\mathbb{T}$, that is, we take $\lambda_s(\dd x) = s\,  \dd x$. The random geometric graph $\operatorname{RGG}(\eta_s, r_s)$ with cut-off radius $r_s>0$ and vertices given by $\eta_s$ is obtained by joining two vertices $x,y \in \eta_s$, $x\neq y$, by an edge if and only if $\|x-y\| \le r_s$, with $\|\cdot\|$ denoting the usual Euclidean norm. In this example, we shall consider the number of vertices with degree less than some given positive integer $k \in \N$ in the dense regime, that is, where the average degree $sr_s^d$ diverges to infinity as $s \to \infty$. Note that the case $k=1$ simply corresponds to the number of isolated vertices in $\operatorname{RGG}(\eta_s,r_s)$. Also note that a vertex having degree less than $k$ is almost surely equivalent to its $k$-nearest neighbour ($k$-NN) distance strictly exceeding $r_s$. Let $N_{<k}(\eta_s, r_s)$ denote the set of vertices with degree less than $k$ in $\operatorname{RGG}(\eta_s,r_s)$. Then $|N_{<k}(\eta_s, r_s)|$, the number of points in $\eta_s$ with degree at most $k-1$ (or having $k$-NN distance larger than $r_s$) can be expressed as a sum of scores as in \eqref{eq:hs} with the score function given by
\begin{align*}
    f_s(x,\eta_s)=\mathds{1}(|B(x,r_s)\cap\eta_s| \le k), \quad x \in \eta_s,
\end{align*}
where $B(x,r)$ denotes the closed ball of radius $r>0$ around a point $x \in \R^d$.
Note that the random variable $|B(x,r_s)\cap\eta_s|$ follows a $\operatorname{Poisson}(\kappa_d s r_s^d)$ distribution, where $\kappa_d$ denotes the volume of the unit ball in $\R^d$, $d \in \N$. Thus, if $sr_s^d \to \infty$ as $s \to \infty$, then
\begin{align}\label{eq:k-iso exp}
    \E[f_s(x,\eta_s+\delta_x)]=\P\left(|B(x,r_s)\cap\eta_s| \le k-1\right) = e^{-\kappa_dsr_s^d}\sum\limits^{k-1}_{n=0}\frac{(\kappa_dsr_s^d)^n}{n!}\asymp s^{k-1}r_s^{d(k-1)}e^{-\kappa_dsr_s^d}.
\end{align}
It also follows that
\begin{equation}\label{eq:rhsA1}
	\int_{\mathbb{T}} \E[f_s(x,\eta_s+\delta_x)]\, \lambda_s(\dd x) \asymp s (sr_s^d)^{k-1} e^{-\kappa_dsr_s^d}.
\end{equation}

For any Poisson process $\eta$ on a measurable space $(\mathbb{X}, \mathcal{X})$ with a $\sigma$-finite intensity measure $\lambda$, a key result to find expectations of certain functionals of $\eta$ is the classical Mecke formula (see e.g.\ \cite[Theorem 4.1]{LastPenrose}): for any measurable function $h : \X \times \mathbf{N} \to [0,\infty]$,
\begin{equation}\label{eq:Mecke}
	\E\int_\X h (x,\eta) \, \eta(\dd x) := \E \sum_{x \in \eta} h(x,\eta) = \E \int_\X  h(x,\eta+\delta_x) \, \lambda(\dd x).
\end{equation}

Let $\mathcal{N}_{k}^t$, $t \ge 0$ denote the set of vertices in $\eta_s^t$ that have at least one neighbour with degree $k$. Drawing a parallel to the example in Section \ref{sec:perc}, these can be viewed as the \textit{pivotal points} for the creation of new vertices with degree at most $k-1$, since removing any vertex from $\mathcal{N}_{k}^t$ would create at least one such new vertex.
\begin{theorem}\label{Thm:KIso}
	 Let $s >0$ and $r_s \in (0,1/2)$ be such that $sr_s^d \to \infty$ as $s \to \infty$. For $k \in \N$, denote the number of vertices with at most $k-1$ neighbours by
	 $$
	 I_{s,k}(\eta_s) = \sum_{x \in \eta_s} \mathds{1}(|B(x,r_s)\cap\eta_s|\le k)
	 $$
	 Then the following holds.
	 \begin{enumerate}[(i)]
	 	\item $I_{s,k}$ is $C\, (sr_s^d)^{-1}$-superconcentrated for some constant $C \in (0,\infty)$, that is,
	$$
	\Var(I_{s,k}(\eta_s)) \le \frac{C}{sr_s^d} \int_{\mathbb{T}} \E (D_x I_{s,k}(\eta_s))^2 \, \lambda_s(\dd x).
	$$
	\item Moreover,
    $$
    \Var(I_{s,k})\asymp \int^\infty_0\E|\mathcal{N}_{k}^0\cap \mathcal{N}_{k}^t| \dd t.
    $$
    \item There exists a constant $C'\in (0,\infty)$ such that the set $\mathcal{N}_{k} \equiv \mathcal{N}_{k}^0$ is $(C' \, \delta_s\, (sr_s^d)^{-1}, \delta_s)$-chaotic for any $\delta_s>0$ with $\delta_s \to 0$ and $\delta_s \,  sr_s^d\to \infty$ as $s \to \infty$.
\end{enumerate}
\end{theorem}
\begin{proof}[\underline{Proof of (i)}]
We prove superconcentration for $I_{s,k}(\eta_s)$ using \autoref{thm:supcon_sos}. To apply the result, we must verify that the scores $f_s(x,\eta_s) = \mathds{1}(|B(x,r_s)\cap\eta_s|\le k)$ satisfy its assumptions. We first consider \ref{A1}. 
When calculating $\E [f_s(x,\eta_s+\delta_x+\delta_y)f_s(y,\eta_s+\delta_x+\delta_y)]$ we consider the three cases (a) $ \|x-y\| > 2r_s$, (b) $\|x-y\| \in (r_s, 2r_s]$ and (c) $\|x-y\| \le r_s$ separately. 
 In case (a), we see $$
\E [f_s(x,\eta_s+\delta_x+\delta_y)f_s(y,\eta_s+\delta_x+\delta_y)] = \E [f_s(x,\eta_s+\delta_x)]\E [f_s(y,\eta_s+\delta_y)],
$$
and thus the contribution from such points $x,y$ to the integral $T_s$ in \ref{A1} is zero.

In case (b), by conditioning on the number of points $0 \le n \le k-1$ lying in $B(x,r_s)\cap B(y,r_s)$, using the independence of Poisson process over disjoint sets, we obtain
\begin{align*}
        \E [f_s(x,\eta_s+\delta_x+\delta_y)&f_s(y,\eta_s+\delta_x+\delta_y)]\\&=\sum\limits_{n=0}^{k-1}\left(e^{-\omega \kappa_dsr_s^d}\frac{(\omega \kappa_dsr_s^d)^n}{n!}\right)\left(\left(e^{-(1-\omega)\kappa_dsr_s^d}\sum\limits_{m=0}^{k-1-n}\Big(\frac{((1-\omega)\kappa_dsr_s^d)^{m}}{m!}\Big)\right)^2\right),
\end{align*}
where $\omega \in (0,1/2)$ represents the proportion of either ball that lies within $B(x,r_s)\cap B(y,r_s)$. This expression considers the probability of $0 \le n \le k-1$ points lying in $B(x,r_s)\cap B(y,r_s)$, then finds the conditional probability of at most $k-1-n$ points lying in the remaining space in each of $B(x,r_s)$ and $B(y,r_s)$. 
Since $sr_s^d \to \infty$ as $s \to \infty$, this yields that there exists a constant $C_1 \in (0,\infty)$ depending only on $d, k$ such that
\begin{align*}
    \E [f_s(x,\eta_s+\delta_x+\delta_y)f_s(y,\eta_s+\delta_x+\delta_y)] \le C_1 e^{-(2-\omega) \kappa_dsr_s^d} (sr_s^d)^{2(k-1)} \le C_1 e^{-1.5 \kappa_dsr_s^d} (sr_s^d)^{2(k-1)}.
\end{align*}
Thus, using $\E [f_s(x,\eta_s+\delta_x)]\E [f_s(y,\eta_s+\delta_y)] \ge 0$ and that $\lim_{s \to \infty}(sr_s^d)^{k} e^{-0.5 \kappa_d s r_s^d} = 0$, we obtain
\begin{align*}
   &\int_{\mathbb{T}}\int_{B(x,2r_s) \setminus B(x,r_s)} \big(\E [f_s(x,\eta_s+\delta_x+\delta_y)f_s(y,\eta_s+\delta_x+\delta_y)] \\
   &\qquad \qquad \qquad - \E [f_s(x,\eta_s+\delta_x)]\E [f_s(y,\eta_s+\delta_y)]\big)\, \lambda_s^2 (\dd y, \dd x) \\
   &\le C_1 (2^d - 1) \kappa_d s^2r_s^d e^{-1.5\kappa_dsr_s^d}(sr_s^d)^{2(k-1)} \\
   &\ll s (sr_s^d)^{k-1}e^{-\kappa_dsr_s^d} \asymp \int_{\mathbb{T}} \E[f_s(x,\eta_s+\delta_x)^2]\, \lambda_s(\dd x).
\end{align*}

Finally for case (c), again using $\E [f_s(x,\eta_s+\delta_x)]\E [f_s(y,\eta_s+\delta_y)] \ge 0$, stationarity of $\eta_s$ yields
\begin{align*}
	&\int_{\mathbb{T}}\int_{B(x,r_s)} \big(\E [f_s(x,\eta_s+\delta_x+\delta_y)f_s(y,\eta_s+\delta_x+\delta_y)] - \E [f_s(x,\eta_s+\delta_x)]\E [f_s(y,\eta_s+\delta_y)]\big)\, \lambda_s^2 (\dd y, \dd x)\\
   & \le s \int_{B(0,r_s)} \E[f_s(0,\eta+\delta_{0}+\delta_{y})f_s(y,\eta+\delta_{0}+\delta_{y})] \, \lambda_s(\dd y)\\
      & \le s \int_{B(0,r_s)} \E[f_s(0,\eta+\delta_{0}+\delta_{y})] \, \lambda_s(\dd y)\\
    &= \kappa_d s^2 r_s^d \, \P( |B(0,r_s) \cap \eta_s| \le k-2) \\
    &=\kappa_d s^2r_s^d\, e^{-\kappa_dsr_s^d}\sum\limits^{k-2}_{n=0}\frac{(\kappa_dsr_s^d)^n}{n!}\\
    & \asymp s \left((sr_s^d)^{k-1} e^{-\kappa_dsr_s^d}\right) \asymp \int_{\mathbb{T}} \E[f_s(x,\eta_s+\delta_x)^2]\, \lambda_s(\dd x),
\end{align*}
where the final step is due to \eqref{eq:rhsA1}. Combining the bounds for the three cases (a) -- (c), we obtain
$$
T_s \le C_2 \int_{\mathbb{T}} \E[f_s(x,\eta_s+\delta_x)^2]\, \lambda_s(\dd x)
$$
for some constant $C_2 \in (0,\infty)$ depending on $d, k$, confirming \ref{A1}.
\medskip

Next, we verify \ref{A2}. By \eqref{eq:k-iso exp}, we know for $y \in \mathbb{T}$,
\begin{equation}\label{eq:A2check}
	\E[f_s(y,\eta_s+\delta_y)^2] \asymp (sr_s^d)^{k-1} e^{-\kappa_dsr_s^d}.
\end{equation}
Further, we observe $D_yf_s(x,\eta_s) \in \{-1,0\}$, since an additional point may reduce the $k$-NN distance of another point, but not increase it. This implies $$
\E\left[\Big(\sum\limits_{x\in \eta_s}D_yf_s(x,\eta_s)\Big)^2\right]\ge \E\left[-\sum\limits_{x\in \eta_s}D_yf_s(x,\eta_s)\right].
$$ 
With the use of the Mecke formula \eqref{eq:Mecke} and noting that $D_yf_s(x,\eta_s) = -1$ exactly when $|B(x,r_s) \cap \eta_s| = k-1$ and $y \in B(x,r_s)$, we see that
\begin{align*}
	\E\left[\Big(\sum\limits_{x\in \eta_s}D_yf_s(x,\eta_s)\Big)^2\right] &\ge
    -\int_{\mathbb{T}} \E[D_yf_s (x,\eta_s+\delta_x)]\, \lambda_s(\dd x)\\
    & = s\int_{B(y,r_s)} \P(|B(x,r_s) \cap \eta_s| = k-1) \, \dd x\\
    &=\kappa_dsr_s^d\left(\frac{(\kappa_dsr_s^d)^{k-1}}{(k-1)!}\right)e^{-\kappa_dsr_s^d}\asymp (sr_s^d)^ke^{-\kappa_dsr_s^d}.
\end{align*}
Thus \eqref{eq:A2check} confirms condition \ref{A2} holds with $\varepsilon_s \asymp (sr_s^d)^{-1}$. The superconcentration is now an immediate consequence of \autoref{thm:supcon_sos} noting that $(1-\sqrt{\eps_s})^{-2} = \Omega(1)$.
\medskip

\noindent\underline{\em Proof of (ii)}. Next we prove the second assertion. We employ results from Section \ref{Sec:VarEst}. We decompose $D_xI_{s,k}$ as in \eqref{eq:Dxdecomp} and the associated variance as in \eqref{eq:vardecomp}. We set $g_1(y,\eta_s)=f_s(y,\eta_s+\delta_y)$ and $g_2(y,\eta_s)=-\sum\limits_{x\in\eta_s}D_yf_s(x,\eta_s)$. For convenience, instead of $g_2(y,\eta_s)$ we use
\begin{align*}
    \bar g_2(y,\eta_s)&=\mathds{1}(\exists x \in \eta_s \cap B(y,r_s) : |\eta_s\cap B(x,r_s)|=k)\asymp g_2(y,\eta_s) =\hspace{-.2cm} \sum_{x \in \eta_s \cap B(y,r_s)} \mathds{1}(|\eta_s\cap B(x,r_s)|=k),
\end{align*}
which holds since $g_2$ is bounded above by a finite positive integer depending only on $d, k$.
Following the structure of \eqref{eq:AxBx}, we define the specific events associated with our geometric graph
\begin{align*}
    M_y:=\{\eta_s:|B(y,r_s)|\cap\eta_s < k\}, \quad N_y:=\{\eta_s:\exists x \in \eta_s \cap B(y,r_s) \text{ with } |\eta_s\cap B(x,r_s)|=k\},
\end{align*}
so that $g_1(y,\eta_s) = \mathds{1}(M_y)$ and $\bar g_2(y,\eta_s)\asymp \bar g_2(y,\eta_s) = \mathds{1}(N_y)$. Similarly as in \eqref{eq:At}, define the sets
\begin{gather*}
\mathcal{M}_k^t=\{y\in\eta_s^t : |B(y,r_s)|\cap\eta_s^t \le k\}, \\
    \mathcal{N}_k^t=\{y\in \eta_s^t:\exists x \in \eta_s^t \cap B(y,r_s) \setminus \{y\} \text{ with } |\eta_s^t \cap B(x,r_s)|=k+1\}.
\end{gather*}
With this and \eqref{eq:vardecomp}, we obtain
\begin{equation}\label{eq:IsVarDecomp}
	\Var(I_{s,k})\asymp \int_0^\infty\mathbb{E}|\mathcal{M}_k^0\cap \mathcal{M}_k^t| \dd t+\int_0^\infty\mathbb{E}|\mathcal{N}_k^0\cap \mathcal{N}_k^t| \dd t - 2\int_0^\infty\mathbb{E}|\mathcal{M}_k^0\cap \mathcal{N}_k^t| \dd t.
\end{equation}
We will find an estimate for $\Var(I_{s,k})$ by first estimating $\int_0^\infty\mathbb{E}|\mathcal{M}_k^0\cap \mathcal{M}_k^t| \dd t$ using \autoref{cor:bddD1}, then estimating $\int_0^\infty\mathbb{E}|\mathcal{N}_k^0\cap \mathcal{N}_k^t| \dd t$ using \autoref{cor:bddD2}, and finally applying \autoref{cor:T3} to show that $\int_0^\infty\mathbb{E}|\mathcal{M}_k^0\cap \mathcal{N}_k^t| \dd t$ is asymptotically negligible compared to the other two summands in \eqref{eq:IsVarDecomp}.
\medskip

Beginning with $\int_0^\infty\mathbb{E}|\mathcal{M}_k^0\cap \mathcal{M}_k^t| \dd t$, we know $g_1(y,\eta_s)\asymp \mathds{1}(M_y)$ and thus it takes the required form for \autoref{cor:bddD1}. In addition, $D_zg_1(y,\eta_s)=D_zf_s(y,\eta_s+\delta_y)\leq0$ for all $y,z \in \mathbb{T}$, and hence the requirement for the second assertion of the corollary is fulfilled. Therefore, arguing as in \eqref{eq:k-iso exp} and using $sr_s^d \to \infty$ as $s \to \infty$ for the final step, we obtain
\begin{align}\label{eq:KisoVarA}
    \int^\infty_0 \mathbb{E} |\mathcal{M}_k^0\cap \mathcal{M}_k^t| \dd t&\leq 2 \int_{\mathbb{T}} \frac{\P(B(y,r_s)\cap\eta_s<k)}{1- (1/4) \log (\P(B(y,r_s)\cap\eta_s<k))}  \lambda_s(\dd y) \nonumber\\
    &\asymp s \frac{s^{k-1}r_s^{d(k-1)}e^{-\kappa_dsr_s^d}}{1- (1/4) \log (s^{k-1}r_s^{d(k-1)}e^{-\kappa_dsr_s^d})}\asymp \frac{s^kr_s^{d(k-1)}e^{-\kappa_dsr_s^d}}{sr_s^d}.
\end{align}

For $\int_0^\infty\mathbb{E}|\mathcal{N}_k^0\cap \mathcal{N}_k^t| \dd t$, note that we have the correct form of function for \autoref{cor:bddD2} as $\bar g_2(y,\eta_s) = \mathds{1}(N_y)$. The other pre-requisite is to show
\begin{align}\label{eq:TrappKIso}
		\mathbb{E} \left[\int_{\mathbb{T}^2} (D_z \bar g_2(y,\eta_s))^2 \lambda_s^2(\dd y,\dd z)\right] \le \alpha  \int_{\mathbb{T}} \P(N_y)\, \lambda_s(\dd y)
\end{align}
for some $\alpha > 0$. Let us proceed with a suitable estimate for $(D_z\bar g_2(y,\eta_s))^2=|D_z\bar g_2(y,\eta_s)|$. Observe that 
\begin{align*}
    D_z\bar g_2(y,\eta_s)=1 \iff &(\not\exists x\in \eta_s\cap B(y,r_s) :|\eta_s\cap B(x,r_s)|=k) \\ &\wedge \big((z\in B(y,r_s)\wedge|\eta_s\cap B(z,r_s)|=k-1)\\&\quad\vee (\exists x\in\eta_s\cap B(y,r_s)\cap B(z,r_s):|\eta_s\cap B(x,r_s)|=k-1 )\big).
\end{align*}
By removing the first condition, we can upper bound its probability as
\begin{align}\label{eq:TrappKIso1}
    \P(D_z\bar g_2(y,\eta_s)=1) 
    &\leq \P(z\in B(y,r_s)\wedge|\eta_s\cap B(z,r_s)|=k-1)) \nonumber\\
    &\qquad + \P(\exists x\in\eta_s\cap B(y,r_s)\cap B(z,r_s):|\eta_s\cap B(x,r_s)|=k-1)\nonumber\\
    &\asymp s^{k-1}r_s^{d(k-1)}e^{-\kappa_dsr_s^d}(\mathds{1}(z\in B(y,r_s)), \nonumber\\
    & +  \P(\exists x\in\eta_s\cap B(y,r_s)\cap B(z,r_s):|\eta_s\cap B(x,r_s)|=k-1),
\end{align}
where the first probability is computed in the final step similarly as \eqref{eq:k-iso exp}. For the second probability in \eqref{eq:TrappKIso1}, we argue noting the finiteness of the set of all $x \in \eta_s\cap B(y,r_s)\cap B(z,r_s)$ with $|\eta_s\cap B(x,r_s)|=k-1$ to obtain
\begin{align}\label{eq:TrappKIso2}
&\P(\exists x\in\eta_s\cap B(y,r_s)\cap B(z,r_s):|\eta_s\cap B(x,r_s)|=k-1) \nonumber\\
& \asymp \E \sum_{x \in \eta_s \cap B(y,r_s)\cap B(z,r_s)} \mathds{1}(|\eta_s\cap B(x,r_s)|=k-1) \nonumber\\
& = s \int_{B(y,r_s)\cap B(z,r_s)} \P(|\eta_s\cap B(x,r_s)| = k-2)\, \dd x \nonumber\\
&\asymp s\operatorname{Vol}(B(y,r_s)\cap B(z,r_s))(s^{k-2}r_s^{d(k-2)}e^{-\kappa_dsr_s^d}),
\end{align}
where the second step is due to the Mecke formula in \eqref{eq:Mecke}. Arguing similarly we also have
\begin{align}\label{eq:TrappKIso3}
     \P(D_z\bar g_2(y,\eta_s)=-1)&=\P((\exists x\in\eta_s\cap B(y,r_s):|\eta_s\cap B(x,r_s)|=k) \nonumber\\
     &\qquad \wedge (z\in B(x,r_s)\forall x\in\eta_s\cap B(y,r_s):|\eta_s\cap B(x,r_s)|=k) \nonumber\\
     &\qquad \wedge \big( (z\in B(y,r_s)\wedge |\eta_s\cap B(z,r_s)|\not=k-1)\vee (z\not\in B(y,r_s)) \big) \nonumber\\
     &\leq\P((\exists x\in\eta_s\cap B(y,r_s):|\eta_s\cap B(x,r_s)|=k) \nonumber\\
     &\qquad \wedge (z\in B(x,r_s)\forall x\in\eta_s\cap B(y,r_s):|\eta_s\cap B(x,r_s)|=k) \nonumber\\
     &\leq\P(\exists x\in\eta_s\cap B(y,r_s)\cap B(z,r_s):|\eta_s\cap B(x,r_s)|=k)\nonumber\\
	&\asymp s^{k-1}r_s^{d(k-1)}e^{-\kappa_dsr_s^d}(s\operatorname{Vol}(B(y,r_s)\cap B(z,r_s))),
\end{align}
where the final step is argued similarly as in \eqref{eq:TrappKIso2}. Combining \eqref{eq:TrappKIso1}, \eqref{eq:TrappKIso2} and \eqref{eq:TrappKIso3}, and using $sr_s^d \to \infty$ as $s\to \infty$, we may calculate an upper bound for the left-hand side of (\ref{eq:TrappKIso}) as
\begin{align*}
    &\mathbb{E} \left[\int_{\mathbb{T}^2} (D_z \bar g_2(y,\eta_s))^2 \,\lambda_s^2(\dd y,\dd z)\right]\\
    &\lesssim \int_{\mathbb{T}^2} s^{k-1}r_s^{d(k-1)}e^{-\kappa_dsr_s^d}(s\operatorname{Vol}(B(y,r_s)\cap B(z,r_s))) \,\lambda_s^2(\dd y,\dd z) \\
    &\qquad +\int_{\mathbb{T}^2} s^{k-1}r_s^{d(k-1)}e^{-\kappa_dsr_s^d}\mathds{1}(z\in B(y,r_s))\, \lambda_s^2(\dd y,\dd z) \\
    & \lesssim  \int_{\mathbb{T}}\int_{(B(y,2r_s)}s^{k+2}r_s^{dk}e^{-\kappa_dsr_s^d}\dd y\dd z+\int_{\mathbb{T}}\int_{(B(y,r_s)}s^{k+1}r_s^{d(k-1)}e^{-\kappa_dsr_s^d} \dd y\dd z \\
    &\asymp\int_{\mathbb{T}}s^{k+2}r_s^{d(k+1)}e^{-\kappa_dsr_s^d}\dd y+\int_{\mathbb{T}}s^{k+1}r_s^{dk}e^{-\kappa_dsr_s^d} \dd y \\
    &=s^{k+2}r_s^{d(k+1)}e^{-\kappa_dsr_s^d}+s^{k+1}r_s^{dk}e^{-\kappa_dsr_s^d}.
\end{align*}
As for the right-hand side of (\ref{eq:TrappKIso}), arguing as for \eqref{eq:TrappKIso2}, we obtain
\begin{align}\label{eq:Eg2}
    \P(N_y) &=\P(\exists x \in \eta_s \cap B(y,r_s) : |\eta_s\cap B(x,r_s)|=k) \nonumber\\
    &\asymp s^{k-1}r_s^{d(k-1)}e^{-\kappa_dsr_s^d}(s \operatorname{Vol}(B(y,r_s))) = (sr_s^d) s^{k-1}r_s^{d(k-1)}e^{-\kappa_dsr_s^d},
\end{align}
so that
\begin{align*}
    \int_{\mathbb{T}} \P(N_y) \, \lambda_s(\dd y)&\asymp s s^{k}r_s^{dk}e^{-\kappa_dsr_s^d} = s^{k+1}r_s^{dk}e^{-\kappa_dsr_s^d}.
\end{align*}
Comparing the orders obtained for $\int_{\mathbb{T}^2} \mathbb{E}[(D_y\bar g_2(x,\eta_s))^2]\,\lambda^2(\dd x,\dd y)$ and $\int_{\mathbb{T}} P(N_y)\,\lambda(\dd x)$, we find that \eqref{eq:TrappKIso} is satisfied with $\alpha=C_3 \,sr_s^d$ for some constant $C_3 \in (0,\infty)$. With this, the requirement for \autoref{cor:bddD2} is fulfilled by $\bar g_2(y,\eta_s)$, and its application yields
\begin{align}\label{eq:KIsoVarB}
    \int^\infty_0 \mathbb{E} |\mathcal{N}_k^0\cap \mathcal{N}_k^t| \dd t &\geq \frac{1}{C_3 sr_s^d+1} \int_{\mathbb{T}} \P(N_y) \, \lambda(\dd x) \nonumber\\
    &\asymp \frac{s^{k+1}r_s^{dk}e^{-\kappa_dsr_s^d}}{sr_s^d}=s^{k}r_s^{d(k-1)}e^{-\kappa_dsr_s^d}.
\end{align}
Comparing \eqref{eq:KisoVarA} and \eqref{eq:KIsoVarB}, we obtain
\begin{align*}
    \int^\infty_0 \mathbb{E} |\mathcal{M}_k^0\cap \mathcal{M}_k^t|\dd t \lesssim \frac{s^kr_s^{d(k-1)}e^{-\kappa_dsr_s^d}}{sr_s^d }\ll s^{k}r_s^{d(k-1)}e^{-\kappa_dsr_s^d} \lesssim \int^\infty_0 \mathbb{E} |\mathcal{N}_k^0\cap \mathcal{N}_k^t| \dd t.
\end{align*}
The second claim now follows from \eqref{eq:IsVarDecomp} upon application of \autoref{cor:T3}.

\medskip

\noindent \underline{\em Proof of (iii)}. For the final claim, we need to show that there exists a constant $C' \in (0,\infty)$ such that for any $\delta_s>0$ with $\delta_s \to 0$ and $sr_s^d \delta_s \to \infty$ as $s \to \infty$, and all $s>0$, $u \ge \delta_s>0$,
\begin{equation*}
	e^{-u}\,\E |\mathcal{N}_k^0 \cap \mathcal{N}_k^u|  \le C' (sr_s^d \delta_s)^{-1} \E |\mathcal{N}_k^0|,
\end{equation*}
Assume that it is false, that is, for each $n \in \N$, there exists a function $\delta^{(n)}_s>0$ in $s$ with $\delta^{(n)}_s \to 0$ and $sr_s^d \delta^{(n)}_s \to \infty$ as $s \to \infty$, some $s_n>0$ and $u_n >\delta^{(n)}_{s_n}$ such that
\begin{equation}\label{eq:assump}
	s_nr_{s_n}^d \delta^{(n)}_{s_n} e^{-u_n} \frac{\E |\mathcal{N}_k^0 \cap \mathcal{N}_k^{u_n}| }{\E |\mathcal{N}_k^0|} >n.
\end{equation}
Since $I_{s,k}$ is $\left(C \,(sr_s^d)^{-1} \right)$-superconcentrated by the first assertion, we have
\begin{align}\label{eq:contra}
s_nr_{s_n}^d \Var(I_{s_n,k}) &\lesssim \int_{\mathbb{T}}\E (D_y I_{s,k})^2 \lambda_s(\dd y) \nonumber\\
&\lesssim \int_{\mathbb{T}} (\E g_1(y,\eta_s)^2 + \E \bar g_2(y,\eta_s)^2) \,  \lambda_s(\dd y)
\lesssim \int_{\mathbb{T}} \E \bar g_2(y,\eta_s)^2 \, \lambda_s(\dd y) \nonumber \\
&=\E\sum_{y \in \eta_s} \mathds{1}(\exists x \in \eta_s \cap B(y,r_s) \setminus \{y\}: |\eta_s\cap B(x,r_s)|=k+1) = \E|\mathcal{N}_k^0|,
\end{align}
where for the third step, we have used $\E g_1^2 \ll \E \bar g_2^2$ which follows from \eqref{eq:k-iso exp} and \eqref{eq:Eg2}, and the penultimate step is an application of the Mecke formula in \eqref{eq:Mecke}. On the other hand, from the second assertion, we obtain
\begin{equation*}
\Var(I_{s_n,k}) \asymp \int_0^\infty \E |\mathcal{N}_k^0 \cap \mathcal{N}_k^t|  \dd t \ge \int_0^{u_n} \E |\mathcal{N}_k^0 \cap \mathcal{N}_k^t|  \dd t.
\end{equation*}
Now arguing as in the proof of \autoref{lem:intermed}, note that
$$
   \E |\mathcal{N}_k^0 \cap \mathcal{N}_k^t|= e^{-t}\int_{\mathbb{T}} \E \bar g_2(x,\eta_s) \bar g_2(x, \eta_s^t) \lambda_s(\dd x).
$$
By a Fock space expansion of $\bar g_2 \in L^2_{\eta_s}$, one can show then that $\E |\mathcal{N}_k^0 \cap \mathcal{N}_k^t|$ is non-increasing in $t$ (see proof of \autoref{thm:eqv} in Section \ref{sec:chaosproof} for a similar argument). Thus, we have $\E |\mathcal{N}_k^0 \cap \mathcal{N}_k^t| \ge \E |\mathcal{N}_k^0 \cap \mathcal{N}_k^{u_n}|$ for $t \in [0,u_n]$. Plugging in the variance lower bound above, from \eqref{eq:assump} we see
$$
s_nr_{s_n}^d \Var(I_{s_n,k}) \gtrsim s_nr_{s_n}^d \int_0^{u_n} \E |\mathcal{N}_k^0 \cap \mathcal{N}_k^t|  \dd t \ge s_nr_{s_n}^d \delta^{(n)}_{s_n}  \E |\mathcal{N}_k^0 \cap \mathcal{N}_k^{u_n}|>n \E |\mathcal{N}_k^0|,
$$
which contradicts \eqref{eq:contra}. This concludes the proof.
\end{proof}

\subsection{$\Gamma$-components in random geometric graphs}\label{sec:isographs}

Consider a random geometric graph $\operatorname{RGG}(\eta_s,r_s)$ with cut-off radius $r_s$ and vertices given by a Poisson point process $\eta_s$ on the $d$-dimensional torus $\mathbb{T}=[0,1]^d/ \sim$, $d \in \N$, with intensity $s>0$ as in Section \ref{sec:isovert}, that is, $\lambda_s(dx) = s \,\dd x$. Let $\Gamma$ be a fixed connected graph on $k$ vertices, $k\geq2$. In this example we shall consider the number of $\Gamma$-components in $\operatorname{RGG}(\eta_s,r_s)$. 
A {\em $\Gamma$-component} is a subgraph of $\operatorname{RGG}(\eta_s,r_s)$ that is isomorphic to $\Gamma$, and is disconnected from the rest of $\eta_s$ (that is the component is `isolated'); see also \cite[Chapter 3]{Penrose} for further discussions, including limit theorems in the sparse ($sr_s^d \to 0$ as $s \to \infty$) and thermodynamic ($sr_s^d$ converges to a constant as $s \to \infty$) regimes. As in Section \ref{sec:isovert}, we consider the dense regime instead, where $sr_s^d$ diverges to infinity. Since all graphs are not feasible as a subgraph of $\operatorname{RGG}(\eta_s,r_s)$ (for instance, a star-shaped graph, with a sufficiently large number of leaf nodes), we assume the component $\Gamma$ we are investigating is feasible in the given setting (see also \cite[Chapter 3]{Penrose}). Let $J_s(\Gamma)$ denote the number of $\Gamma$-components of $\operatorname{RGG}(\eta_s,r_s)$; it can be expressed as a sum of scores with the score function given by $f_s(x_{1:k},\eta_s)/k!$ where
\begin{equation}\label{eq:graphscore}
    f_s(x_{1:k},\eta_s)\equiv f_s(x_1,\dots,x_k,\eta_s)=\mathds{1}(x_{1:k} \cong \Gamma \wedge d(x_{1:k},\eta_s\setminus x_{1:k})> r_s)
\end{equation}
for $x_1,\dots,x_k \in \eta_s$ distinct. Here $x_{a:b}=\{x_a,\dots,x_b\}$, the symbol $\cong$ stands for graph isomorphism, and for two finite point collections $X, Y \subset \mathbb{T}$, we write 
$$
d(X,Y):= \min_{x\in X, y \in Y} \|x-y\|.
$$
Moving forward, we will also denote $\eta_s^X = \eta_s + \delta_X$ for any finite set $X \subset \mathbb{T}$, and write $\eta_s^y\equiv \eta_s^{\{y\}}$, $y \in \mathbb{T}$ for notational simplicity. For $k \in \N$, and a locally finite point configuration $\mu$ in $\mathbb{T}$, the set $\mu^{\neq,k}$ stands for the collection of all ordered $k$-tuples of distinct points from $\mu$, that is, 
$$
\mu^{\neq,k}:=\{(x_1,\dots,x_k) : x_1, \hdots,x_k \in \mu, x_i \neq x_j \text{ for $i \neq j$}\}.
$$ 

\begin{theorem}\label{Thm:GammaComps}
	 Let $s >0$ and $r_s \in (0,1/2)$ be such that $sr_s^d \to \infty$ as $s \to \infty$. For $k \in \N$ and a fixed feasible connected graph $\Gamma$ on $k$-vertices, denote the number of $\Gamma$-components in $RGG(\eta_s,r_s)$ by
	 $$
	 J_s(\Gamma)=\frac{1}{k!}\sum_{(x_1,\dots,x_k)\in\eta_s^{\neq,k}}f_s(x_{1:k},\eta_s).
	 $$
	 Then
     $J_s$ is $C\, (sr_s^d)^{-1}$-superconcentrated for some constant $C \in (0,\infty)$ depending only on $d$ and $k$, that is,
	$$
	\Var(J_s(\eta_s)) \le \frac{C}{sr_s^d} \int_{\mathbb{T}} \E (D_x J_s(\eta_s))^2 \, \lambda_s(\dd x).
	$$ 
\end{theorem}

    Note that when counting $\Gamma$-components, we do not count permutations of an ordered $k$-tuple $(x_1, \hdots,x_k)$ that is isomorphic to $\Gamma$, which results in the additional $1/k!$ factor in $J_s$. We also note here by \autoref{thm:eqv}, for any $\delta_s>0$ with $\delta_s \to 0$ and $\delta_s \, sr_s^d\to \infty$ as $s \to \infty$, one obtain that $J_s$ is $(C' (\delta_s \, sr_s^d)^{-1}, \delta_s)$-chaotic for some $C'\in (0,\infty)$. Moreover, carrying out a similar analysis as done in Theorems \ref{thm:crossing} and \ref{Thm:KIso} here should also yield a chaotic set in the context of $\Gamma$-components.

    In the remainder of this section, it will often be more convenient to consider $G_s(\Gamma)=k!J_s(\Gamma)$. We will prove $J_s(\Gamma)$ to be superconcentrated using a similar method to how we proved the superconcentration of the number of vertices with maximum degree less than $k$ in Section \ref{sec:isovert}. Note that we cannot directly apply \autoref{thm:supcon_sos} since $f_s$ in this example is a multivariate score function. It is worth noting that a multivariate version of \autoref{thm:supcon_sos} is possible to prove, however, because of notational complexity of such a result, we have chosen here to prove superconcentration directly. 

We start by estimating the order of the variance of $J_s$. To that end, we require the following lemma, which provides the core estimate needed to verify the multivariate analogue of condition \ref{A1} in \autoref{thm:supcon_sos}. For $k \in \N$ and $s>0$, denote
    \begin{equation*}
        T_{s,k} := \int\limits_{\substack{\mathbb{T}^{k}}}\int\limits_{\substack{\mathbb{T}^{k} \\ d(x_{1:k},x_{k+1:2k})\in(r_s,2r_s]}} \E [f_s(x_{1:k},\eta_s^{x_{1:2k}})f_s(x_{k+1:2k},\eta_s^{x_{1:2k}})]  \, \lambda_s^{k}(\dd x_{1:k}) \lambda_s^{k}(\dd x_{k+1:2k}).
    \end{equation*}

\begin{lemma} \label{lem:GammaComps}
    For $f_s$ as in \eqref{eq:graphscore}, it holds that
    \begin{equation*}
        T_{s,k} \leq C \int_{\mathbb{T}^k}\E[f_s(x_{1:k},\eta_s^{x_{1:k}})^2] \, \lambda_s^k(\dd x_{1:k})
    \end{equation*}
    for some constant $C \in (0,\infty)$ depending only on $d$ and $k$.
\end{lemma}

\begin{proof}
    Let $CH(x_{1:2k})$ denote the the convex hull of $x_{1:2k}$ and $E(x_{1:2k})$ its extreme points. Thus, $E(x_{1:2k}) \subseteq x_{1:2k}$. We may split $T_{s,k}$ into two separate cases - the case where at least one point in $E(x_{1:2k})$ is in $x_{k+1:2k}$, and the other case, where all points in $E(x_{1:2k})$ are in  $x_{1:k}$.
    We thus have $T_{s,k} = I_1 + I_2$, where
    \begin{multline*}
        I_1= \int\limits_{\substack{\mathbb{T}^{k}}}\; \lambda_s^{k}(\dd x_{1:k}) \hspace{-1cm}\int\limits_{\substack{\mathbb{T}^{k} \\ d(x_{1:k},x_{k+1:2k})\in(r_s,2r_s], \\ x_{k+1:2k} \cap E(x_{1:2k}) \neq \emptyset}} \E [f_s(x_1,\dots,x_k,\eta_s^{x_{1:2k}})f_s(x_{k+1},\dots,x_{2k},\eta_s^{x_{1:2k}})]\, \lambda_s^{k}(\dd x_{k+1:2k}),
    \end{multline*}
    and
    \begin{multline*}
        I_2= \int\limits_{\substack{\mathbb{T}^{k}}}\; \lambda_s^{k}(\dd x_{k+1:2k}) \hspace{-1cm}\int\limits_{\substack{\mathbb{T}^{k} \\ d(x_{1:k},x_{k+1:2k})\in(r_s,2r_s], \\ E(x_{1:2k}) \subseteq x_{1:k}}} \E [f_s(x_1,\dots,x_k,\eta_s^{x_{1:2k}})f_s(x_{k+1},\dots,x_{2k},\eta_s^{x_{1:2k}})]\, \lambda_s^{k}(\dd x_{1:k}).
    \end{multline*}

    First, we consider $I_1$. In this case, there is at least one extreme point in $E(x_{1:2k})$ from $x_{k+1:2k}$. Since we additionally have $d(x_{1:k}, x_{k+1:2k}) > r_s$, there exists a constant $c \in (0,1/2)$ depending only on $d$ such that (see Figure \ref{fig:cover})
    \begin{equation}\label{eq:vollb}
        \operatorname{Vol}\left(\bigcup\limits_{i=1}^{2k}B(x_i,r_s)\setminus\bigcup\limits_{i=1}^{k}B(x_i,r_s)\right)\geq c \,\kappa_d sr_s^d.
    \end{equation}

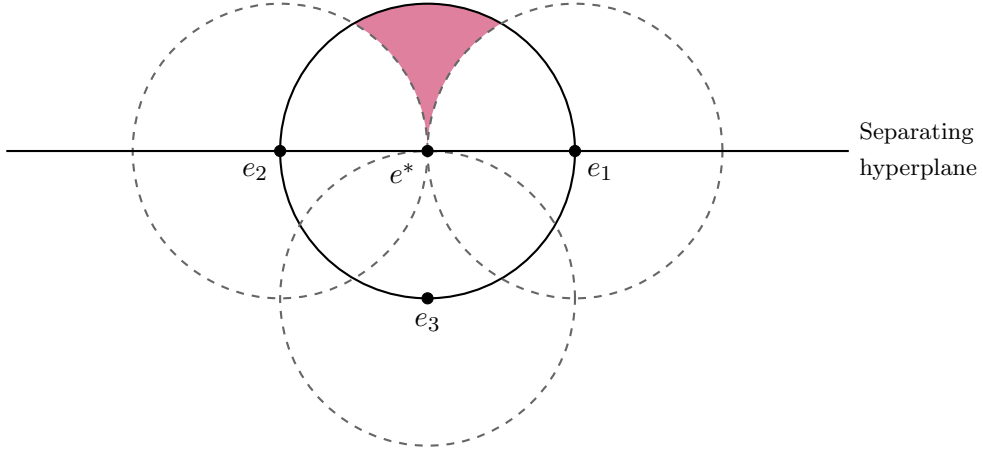
\begin{figure}
\begin{center}
    \begin{tikzpicture}[scale=1.4]
    \def\rs{1.4}

    \begin{scope}
        \clip (0,0) circle (\rs);
        \fill[purple!50] (-3,0) rectangle (3,3);
        
        \fill[white] (\rs,0) circle (\rs);   
        \fill[white] (-\rs,0) circle (\rs);  
        \fill[white] (0,-\rs) circle (\rs);  
    \end{scope}

    \draw[thick] (0,0) circle (\rs);
    \draw[thick, dashed, black!60] (\rs,0) circle (\rs);
    \draw[thick, dashed, black!60] (-\rs,0) circle (\rs);
    \draw[thick, dashed, black!60] (0,-\rs) circle (\rs);

    \draw[thick] (-4,0) -- (4,0) node[right, align=left] {\footnotesize Separating\\\footnotesize hyperplane};

    \filldraw (0,0) circle (1.5pt) node[below left=1pt] {$e^*$};
    \filldraw (\rs,0) circle (1.5pt) node[below right=1pt] {$e_1$};
    \filldraw (-\rs,0) circle (1.5pt) node[below left=1pt] {$e_2$};
    \filldraw (0,-\rs) circle (1.5pt) node[below=2pt] {$e_3$};
\end{tikzpicture}
\caption{Illustration of the lower bound in \eqref{eq:vollb} in 2D. The point $e^*$ is an extreme point in $x_{k+1:2k}$, while $e_1,e_2,e_3$ are points in $x_{1:k}$. Since $x_{1:k}$ lie below the separating hyperplane and $d(x_{1:k},e^*) \ge r_s$, even in the worst case $\cup_{x \in x_{1:k}} B(x,r_s)$ cannot cover the coloured region, which has volume $\Omega(r_s^d)$.} \label{fig:cover}
\end{center}
\end{figure}
Thus, we can upper bound
    \begin{align*} 
        I_1& =  \int\limits_{\substack{\mathbb{T}^{k}}}\; \lambda_s^{k}(\dd x_{1:k}) \hspace{-1cm}\int\limits_{\substack{\mathbb{T}^{k} \\ d(x_{1:k},x_{k+1:2k})\in(r_s,2r_s], \\ x_{k+1:2k} \cap E(x_{1:2k}) \neq \emptyset}} \mathds{1}(x_{1:k} \cong \Gamma, x_{k+1:2k} \cong \Gamma)\P(d(x_{1:2k}, \eta_s)> r_s)\,  \lambda_s^{k}(\dd x_{k+1:2k})\\
        &= \int\limits_{\substack{\mathbb{T}^{k}}} \mathds{1}(x_{1:k} \cong \Gamma) \P(d(x_{1:k}, \eta_s)> r_s) \, \lambda_s^{k}(\dd x_{1:k}) \\
        & \qquad \times \hspace{-1.2cm}\int\limits_{\substack{\mathbb{T}^{k} \\ d(x_{1:k},x_{k+1:2k})\in(r_s,2r_s], \\ x_{k+1:2k} \cap E(x_{1:2k}) \neq \emptyset}} \mathds{1}(x_{k+1:2k} \cong \Gamma)\P(d(x_{k+1:2k}, \eta_s)> r_s \,\big|\, d(x_{1:2k}, \eta_s)> r_s)\, \lambda_s^{k}(\dd x_{k+1:2k})\\
        &\leq\int\limits_{\substack{\mathbb{T}^{k}}}\E[f_s(x_{1:k},\eta_s^{x_{1:k}})^2] \, \lambda_s^k(\dd x_{1:k})\\
        & \qquad \times \int\limits_{\substack{\mathbb{T}^{k} \\d(x_{1:k},x_{k+1:2k})\in(r_s,2r_s]}} \mathds{1}(x_{k+1:2k} \cong \Gamma) e^{-s\operatorname{Vol}(\bigcup_{j=k+1}^{2k}B(x_j,r_s)\setminus \bigcup_{i=1}^{k}B(x_i,r_s))} \,\lambda_s^k(\dd x_{k+1:2k}) \\
        &\lesssim e^{-c \kappa_dsr_s^d} s^k r_s^{kd} \int_{\mathbb{T}^k}\E[f_s(x_{1:k},\eta_s^{x_{1:k}})^2]\,\lambda_s^k(\dd x_{1:k}) \lesssim \int_{\mathbb{T}^k}\E[f_s(x_{1:k},\eta_s^{X_{1:k}})^2]\,\lambda_s^k(\dd x_{1:k}).
    \end{align*}
   Here in the penultimate step, the factor $r_s^{kd}$ is, up to constants, an upper bound to the volume of the region for $x_{k+1:2k}$ such that $d(x_{1:k}, x_{k+1:2k}) \in (r_s,2r_s]$ with $x_{k+1:2k}$ isomorphic to the connected graph $\Gamma$,  while the last step is true since $e^{-c\kappa_dsr_s^d}(sr_s^d)^k$ approaches 0 as $s\to\infty$ in the dense regime.
    
\medskip

    Next we turn our attention to $I_2$. Since $\emptyset \neq E(x_{1:2k}) \subseteq x_{1:k}$, arguing as in \eqref{eq:vollb}, we have that
    $$
    \operatorname{Vol}(\bigcup\limits_{i=1}^{2k}B(x_i,r_s)\setminus\bigcup\limits_{i=k+1}^{2k}B(x_i,r_s))\geq c\kappa_d sr_s^d.
    $$ 
    Thus, arguing similarly as above, we obtain
        \begin{align*} 
        I_2& =  \int\limits_{\substack{\mathbb{T}^{k}}}\; \lambda_s^{k}(\dd x_{k+1:2k}) \hspace{-1cm}\int\limits_{\substack{\mathbb{T}^{k} \\ d(x_{1:k},x_{k+1:2k})\in(r_s,2r_s], \\ E(x_{1:2k}) \subseteq x_{1:k}}}\mathds{1}(x_{1:k} \cong \Gamma, x_{k+1:2k} \cong \Gamma)\P(d(x_{1:2k}, \eta_s)> r_s)\, \lambda_s^{k}(\dd x_{1:k}) \\
        &\leq\int\limits_{\substack{\mathbb{T}^{k}}}\E[f_s(x_{k+1:2k},\eta_s^{x_{k+1:2k}})^2] \, \lambda_s^k(\dd x_{k+1:2k})\\
        & \qquad \times \int\limits_{\substack{\mathbb{T}^{k} \\ d(x_{1:k},x_{k+1:2k})\in(r_s,2r_s]}} \mathds{1}(x_{1:k} \cong \Gamma) e^{-s\operatorname{Vol}(\bigcup_{i=1}^{k}B(x_i,r_s)\setminus \bigcup_{j=k+1}^{2k}B(x_j,r_s))}\, \lambda_s^k(\dd x_{1:k}) \\
        &\lesssim e^{-c\kappa_dsr_s^d}(sr_s^d)^k \int_{\mathbb{T}^k}\E[f_s(x_{k+1:2k},\eta_s^{x_{k+1:2k}})^2]\, \lambda_s^k(\dd x_{k+1:2k}) \\
        &\lesssim \int_{\mathbb{T}^k}\E[f_s(x_{k+1:2k},\eta_s^{x_{k+1:2k}})^2]\, \lambda_s^k(\dd x_{k+1:2k})
        =  \int_{\mathbb{T}^k}\E[f_s(x_{1:k},\eta_s^{x_{1:k}})^2]\, \lambda_s^k(\dd x_{1:k}).
    \end{align*}
    Since $T_{s,k} = I_1+I_2$, combining the bounds above yields the result.
\end{proof}
Now we are ready to prove \autoref{Thm:GammaComps}. Here we also require a multivariate version of the Mecke formula \eqref{eq:Mecke}. For every $k\in\N$ and any measurable function $h:\mathbb{T}^{k} \times\mathbf{N}\to [0,\infty]$, it holds that (see \cite[Theorem 4.4]{LastPenrose})
\begin{equation}\label{eq:Mecke_multivar}
	\E \sum_{(x_1, \hdots, x_k) \in \eta_s^{\neq,k}} h(x_{1:k},\eta) 
	= \E \int_{\mathbb{T}^k}  h(x_{1:k},\eta_s^{x_{1:k}}) \, \lambda_s^k(\dd x_{1:k}).
\end{equation}
\begin{proof}[Proof of \autoref{Thm:GammaComps}]
    Recall that $G_s(\Gamma)=k!J_s(\Gamma)$. Observe that
    \begin{align}\label{eq:g}
    D_yG_s(\Gamma)&=\sum_{(x_1,\dots,x_k)\in (\eta_s^y)^{\neq,k}} f_s(x_{1:k},\eta_s^y)-\sum_{(x_1,\dots,x_k)\in\eta_s^{\neq,k}} f_s(x_{1:k},\eta_s) \nonumber\\
    &= \sum\limits_{\substack{(x_1,\dots,x_k)\in (\eta_s^y)^{\neq,k} \\ y\in x_{1:k}}}f_s(x_{1:k},\eta_s^y) +\sum_{(x_1,\dots,x_k)\in\eta_s^{\neq,k}} D_yf_s(x_{1:k},\eta_s) \nonumber\\
    &=: g_1(y,\eta_s) - g_2(y,\eta_s).
\end{align}
The Poincar\'e inequality (\ref{thm:poincare}) thus yields
\begin{equation} \label{eq:var_Gs}
	\Var(G_s(\Gamma))\leq\E \int_\X (D_y G_s(\Gamma))^2 \, \lambda_s(\dd y)
	= \int _\X\E [g_1(y,\eta_s)^2+g_2(y,\eta_s)^2-2g_1(y,\eta_s)g_2(y,\eta_s)] \,\lambda_s(\dd y).
\end{equation}

Now, we will show $\E g_1(y,\eta_s)^2  \le \eps_s \E g_2(y,\eta_s)^2$ for $\lambda_s$-a.e.\ $y \in \X$ for some $\eps_s \to 0$ as $s\to \infty$ (this is equivalent to showing \ref{A2} in the single variable case in \autoref{thm:supcon_sos}).

We begin by calculating $\E g_2(y,\eta_s)^2$. Observe that $D_y f_s(x_{1:k},\eta_s) \in \{-1,0\}$, since the additional point $y$ may cause a $\Gamma$-component at $x_{1:k}$ to no longer be isolated, but cannot cause $x_{1:k}$ to form a $\Gamma$-component when it did not before. Thus for $g_2$ as in \eqref{eq:g}, we trivially lower bound
\begin{equation}\label{eq:g2bd1}
\E g_2(y,\eta_s)^2 \ge \E\left[-\sum_{(x_1,\dots,x_k)\in\eta_s^{\neq,k}}D_yf_s(x_{1:k},\eta_s)\right].
\end{equation}
Note that for distinct $x_1, \hdots, x_k \in \mathbb{T}$,
\begin{align*}
    \left\{D_yf_s(x_{1:k},\eta_s^{x_{1:k}}) = -1 \right\} \iff & \left\{x_{1:k} \cong \Gamma\right\} \cap \left\{d(x_{1:k}, \eta_s)>r_s\right\} \cap \left\{y \in \cup_{i=1}^k B(x_i,r_s)\right\}
\end{align*} 
so that
\begin{equation}\label{eq:ExpD-1}
    -\E \left[D_yf_s(x_{1:k},\eta_s^{x_{1:k}})\right] \ge \mathds{1}(x_{1:k} \cong \Gamma) \mathds{1}\left(x_1 \in B(y,r_s)\right) \P \left(d(x_{1:k}, \eta_s)>r_s\right).
\end{equation}
Plugging in \eqref{eq:ExpD-1}, an application of \eqref{eq:Mecke_multivar} and \eqref{eq:g2bd1} yields
\begin{align}\label{eq:g2bd}
    \E g_2(y,\eta_s)^2 &\ge 
    -\int_{\mathbb{T}^k} \E[D_yf_s(x_{1:k},\eta_s^{x_{1:k}})]\, \lambda_s^k(\dd x_{1:k})\nonumber\\
    &\ge s^k\int_{\mathbb{T}^{k-1}}\int_{B(y,r_s)} \mathds{1}(x_{1:k} \cong \Gamma) \P \left(d(x_{1:k}, \eta_s)>r_s\right) \,\dd x_1 \dd x_{2:k} \nonumber\\
    &= \kappa_dr_s^d s^k \int_{\mathbb{T}^{k-1}} \mathds{1}(\{y,x_{2:k}\} \cong \Gamma ) \P \left(d(\{y,x_{2:k}\}, \eta_s)>r_s\right)\,\dd x_{2:k},
\end{align}
where the last step holds due to the stationarity of $\eta_s$.

Next, let us estimate $\E g_1(y,\eta_s)^2$ with $g_1$ as in \eqref{eq:g}. Since $y$ is a common point in all the $k$-tuples in the sum $g_1$ defined in \eqref{eq:g}, we can have $f_s=1$ for at most one such set of $k$ points. Counting its $k!$ permutations (when there is such a set), we thus have that $g_1(y,\eta_s) \in\{0,k!\}$. Therefore 
\begin{equation}\label{eq:g1sqbd}
    \E g_1(y,\eta_s)^2=k!\, \E g_1(y,\eta_s).
\end{equation}
Now, with the use of \eqref{eq:Mecke_multivar} and noting that $f_s$ is symmetric in $x_{1:k}$, for $s$ large enough we can evaluate 
\begin{align}\label{eq:g1bd}
    \E g_1(y,\eta_s)
    &=k \int_{\mathbb{T}^{k-1}}\E[f_s(\{y ,x_{2:k}\},\eta_s^{\{y,x_{2:k}\}})] \, \lambda_s^{k-1}(\dd x_{2:k}) \nonumber\\
    &=k s^{k-1}\int_{\mathbb{T}^{k-1}} \mathds{1}(\{y,x_{2:k}\} \cong \Gamma ) \P \left(d(\{y,x_{2:k}\}, \eta_s)>r_s\right)\,\dd x_{2:k}.
\end{align}
With both the expectations estimated, comparing \eqref{eq:g2bd} and \eqref{eq:g1bd}, along with the relation $\E g_1(y,\eta_s)^2=k! \E g_1(y,\eta_s)$ yields
\begin{equation}\label{eq:A2mult}
    \E g_1(y,\eta_s)^2  \le \frac{k k!}{\kappa_d sr_s^d} \E g_2(y,\eta_s)^2 = :\eps_s \E g_2(y,\eta_s)^2,
\end{equation}
with $\eps_s = k k!(\kappa_d sr_s^d )^{-1} \asymp (sr_s^d)^{-1} \to 0$ as $s \to \infty$.

We can now combine \eqref{eq:A2mult} and \autoref{lem:GammaComps} to obtain the result. Note by the Cauchy-Schwarz inequality and \eqref{eq:A2mult} that for $\lambda_s$-a.e.\ $y \in \X$,
$$
\E[g_1(y,\eta_s)g_2(y,\eta_s)]\leq\sqrt{{\E[g_1(y,\eta_s)^2]\E[g_2(y,\eta_s)^2]}}\le \sqrt{\eps_s} \E[g_2(y,\eta_s)^2].
$$
Combining this and \eqref{eq:A2mult} with the variance bound given in \eqref{eq:var_Gs}, we see
\begin{equation}\label{eq:PoincareRHS}
\E \int_\X (D_yJ_s(\Gamma))^2 \, \lambda_s(\dd x)\ge (1-\sqrt{\eps_s})^2 \int_\X \E [g_2(y,\eta_s)^2] \,\lambda_s(\dd y).
\end{equation}
Thus, the Poincar\'e inequality upper bound is of the order $\E [g_2(y,\eta_s)^2] \,\lambda_s(\dd y)$.

\medskip

To show superconcentration, we next compute a tight bound for the variance without using the Poincar\'{e} inequality. By \eqref{eq:Mecke_multivar}, we can expand
\begin{align}\label{eq:varest_comgrph}
   \Var(G_s)&=\E \Big(\sum_{(x_1,\dots,x_k)\in\eta_s^{\neq,k}}f_s(x_{1:k},\eta_s)\Big)^2-(\E G_s)^2 \nonumber\\
	&=\E \Big[k! \sum\limits_{(x_1,\dots,x_k)\in \eta_s^{\neq,k}}f_s^2(x_{1:k},\eta_s^{x_{1:k}}) 
   +\sum\limits_{(x_1,\dots,x_{2k})\in\eta_s^{\neq,2k}}[f_s(x_{1:k},\eta_s^{x_{1:2k}})f_s(x_{k+1:2k},\eta_s^{x_{1:2k}})] \Big]\nonumber\\
   &\qquad \qquad \qquad -\left[\int_{\mathbb{T}^k} \E[f_s(x_{1:k},\eta_s^{x_{1:k}})]\, \lambda_s^k(\dd x_{1:k})\right]^2 \nonumber\\
   &=k! \int_{\mathbb{T}} \E[f_s(x_{1:k},\eta_s^{x_{1:k}})^2] \,\lambda_s^{k}(\dd x_{1:k})\nonumber\\
   &\qquad+\int_{\mathbb{T}^k}\int_{\mathbb{T}^k} \Big(\E [f_s(x_{1:k},\eta_s^{x_{1:2k}})f_s(x_{k+1:2k},\eta_s^{x_{1:2k}})] \nonumber\\
        & \qquad \qquad\qquad- \E [f_s(x_{1:k},\eta_s^{x_{1:k}})]\E[f_s(x_{k+1:2k},\eta_s^{x_{k+1:2k}})]\Big) \, \lambda_s^{k}(\dd x_{1:k}) \lambda_s^{k}(\dd x_{k+1:2k})\nonumber\\
    &=:k! \int_{\mathbb{T}} \E[f_s(x_{1:k},\eta_s^{x_{1:k}})^2] \,\lambda_s^{k}(\dd x_{1:k})+ I.
\end{align}
In the second equality above, many terms in the expansion of the first square vanish, since they are products of the score function $f_s$ at two $k$-tuples that share at least one point, and therefore cannot both form an isolated $\Gamma$-component in $\eta_s$.

\noindent To compute the integral $I$ in the above variance expansion \eqref{eq:varest_comgrph}, as in Section \ref{sec:isovert}, we consider the following three cases: (a) $ d(x_{1:k},x_{k+1:2k}) > 2r_s$, (b) $d(x_{1:k},x_{k+1:2k}) \in (r_s, 2r_s]$ and (c) $d(x_{1:k},x_{k+1:2k}) \le r_s$. In case (a), due to spatial independence, we have 
\begin{align*}
    \E [f_s(x_{1:k},\eta_s^{x_{1:2k}})f_s(x_{k+1:2k},\eta_s^{x_{1:2k}})]=\E [f_s(x_{1:k},\eta_s^{x_{1:k}})]\E[f_s(x_{k+1:2k},\eta_s^{x_{k+1:2k}})],
\end{align*}
and thus case (a) does not contribute to the integral $I$.

\noindent Also in case (c), 
$$
\E [f_s(x_{1:k},\eta_s^{x_{1:2k}})f_s(x_{k+1:2k},\eta_s^{x_{1:2k}})]=0,
$$
as in this case, the two components necessarily have connections between themselves and thus they cannot be isolated. Thus, the contribution from this case to the integral $I$ is non-positive.

\noindent Finally, note that the integral $I$ in case (b), upon dropping the non-positive summand, is bounded above by $T_{s,k}$ in \autoref{lem:GammaComps}. Combining these bounds for cases (a) -- (c) we obtain
$$
I \le T_{s,k} \leq C \int_{\mathbb{T}^k}\E[f_s(x_{1:k},\eta_s^{x_{1:k}})^2]\lambda_s^k(\dd x_{1:k})
$$
with $C \in (0,\infty)$ as in \autoref{lem:GammaComps}.
Inserting this into our overall variance expansion \eqref{eq:varest_comgrph} yields

\begin{align*}
    \Var(G_s)&\le (k!+C) \int_{\mathbb{T}^k}\E[f_s(x_{1:k},\eta_s^{x_{1:k}})^2]\, \lambda_s^k(\dd x_{1:k}) \\
    & = \frac{(k!+C)}{k!k} \int_{\mathbb{T}} k! \left(k \int_{\mathbb{T}^{k-1}}\E[f_s(\{y ,x_{2:k}\},\eta_s^{\{y,x_{2:k}\}})] \, \lambda_s^{k-1}(\dd x_{2:k})\right) \lambda_s(\dd y)\\
    &= \frac{(k!+C)}{k!k} \int_{\mathbb{T}} \E[g_1(y,\eta_s)^2]\, \lambda_s (\dd y) \le \frac{(k!+C)}{k!k} \eps_s \int_{\mathbb{T}} \E[g_2(y,\eta_s)^2]\, \lambda_s (\dd y),
\end{align*}
where the penultimate step is due to the first equality in \eqref{eq:g1bd} and \eqref{eq:g1sqbd}, and the final step uses \eqref{eq:A2mult}. 
Finally, applying the inequality \eqref{eq:PoincareRHS}, we obtain the bound
$$
\Var(G_s) \le \frac{(k!+C)\eps_s}{k}  \int_{\mathbb{T}} \E[g_2(y,\eta_s)^2]\, \lambda_s (\dd y) \le \frac{(k!+C)}{k! k} \frac{\eps_s}{(1-\sqrt{\eps_s})^2 } \E \int_\X (D_yJ_s(\Gamma))^2 \, \lambda_s(\dd x).
$$
Noting that $\Var(J_s) = (k!)^{-2} \Var(G_s)$ and $\eps_s = k k!(\kappa_d sr_s^d )^{-1} \asymp (sr_s^d)^{-1} \to 0$ as $s \to \infty$ yields the conclusion.
\end{proof}

\section{Proofs of main results}\label{sec:proofs} 
In this section, we prove the results in Section \ref{sec:mainresults}. We first collect in Section \ref{sec:PoiBG} some basic properties of the Poisson space which will be essential for our analysis. Sections \ref{sec:chaosproof}, \ref{sec:varproof} and  \ref{sec:supproof} then contain the proofs of the results from Sections \ref{sec:chaosresults}, \ref{Sec:VarEst} and \ref{sec:SCres}, respectively. Throughout this section, for a $\sigma$-finite measure $\lambda$ on a measurable space $(\mathbb{X}, \mathcal{X})$, we denote the scalar product in $L^2(\lambda^m)$, $m \in \N$ by $\langle \cdot,\cdot\rangle_m$ and the associated norm by $\|\cdot\|_m$, while recall that $\|\cdot\|_{L_\eta^p}$, $p >0$ denotes the $L^p$-norm for  functionals of a Poisson process $\eta$.

\subsection{Poisson space preliminaries}\label{sec:PoiBG}

Let $\eta$ be a Poisson process on a measurable space $(\mathbb{X}, \mathcal{X})$ with a $\sigma$-finite intensity measure $\lambda$. Recall the add-one cost operator $D_x(\eta)$, $x \in \X$ defined in \eqref{eq:addone}, and inductively define the iterated add-one cost operator as
$$
D_{x_1, \hdots,x_n}^n F(\eta) = D_{x_1} (D^{n-1}_{x_2, \hdots, x_n} F(\eta)), \quad n \ge 2, x_1, \hdots, x_n \in \X.
$$

Every square-integrable Poisson functional $F \in L^2_\eta$ admits the so-called Fock space representation (see \cite[Equation (18.20)]{LastPenrose}) as an orthogonal expansion: 
\begin{equation}\label{eq:FS}
	F (\eta) = \E F + \sum_{n=1}^\infty I_n(f_n) = \sum_{n=0}^\infty I_n(f_n),
\end{equation}
where as a convention, we write $f_0 = \mathbb{E}[F(\eta)]$ and for $n \ge 1$, we denote $f_n(x_1, \hdots, x_n) = \frac{1}{n!} \E[D_{x_1,\hdots,x_n}^n F(\eta)]$ which is symmetric. The functions $I_n$ are the so-called Wiener-Ito integrals (see \cite[Chapter 12]{LastPenrose} for a definition) satisfying for all symmetric functions $f \in L^2(\lambda^m), g \in L^2(\lambda^n)$,
\begin{equation}\label{eq:Iorth}
	\E [I_m(f) I_n(g)] = \mathds{1}\{m=n\} m! \int fg \, \dd\lambda^m =  \mathds{1}\{m=n\} m! \langle f,g \rangle_m.
\end{equation}

We write $F \in \operatorname{dom} D$ if $F \in L^2_\eta$ and $\int_\X \E (D_x F)^2 \lambda(\dd x)<\infty$, that is, $DF \in L^2 (\lambda \otimes \P_\eta)$. For such an $F$ with the expansion \eqref{eq:FS}, one can also expand $D_x F$ to obtain
\begin{equation}\label{eq:Dexp}
	D_x F = \sum_{n=1}^\infty n I_{n-1}(f_n(x,\cdot)).
\end{equation}

Recall the semigroup of operators $(P_t)_{t \ge 0}$ defined in \eqref{eq:Meh}. It satisfies a commutation relation with the operator $D$ given by
\begin{equation}\label{eq:DPcomm}
	D_x(P_t F) = e^{-t} P_t (D_xF), \qquad \text{$\lambda$-a.e.\ $x \in \X, \P_\eta$-a.s., $t \ge 0$}.
\end{equation}
The following contractive property holds (see \cite[Chapter 1, equation (74)]{PeccatiReitzner}) for all $F \in L^2_\eta$ and $p \ge 1$:
\begin{equation}\label{eq:contractive}
	\E |P_t F|^p \le \E |F|^p, \qquad t \ge 0.
\end{equation}
While the semigroup is not generally hypercontractive, one has \textit{restricted hypercontractivity} (see \cite[Theorem 1.4]{NourdinPeccatiYang}): if $F \ge 0$ and $DF \le 0$, then for all $t \ge 0$ and $p \ge 1$,
\begin{equation}\label{eq:resthyp}
	\left\|P_t F\right\|_{L_\eta^{1+(p-1) e^t}} \leq \|F\|_{L_\eta^p}.
\end{equation}
Moreover, for $F \in L^2_\eta$ having the expansion given in \eqref{eq:FS}, one also has the orthogonal decomposition
\begin{equation}\label{eq:Ptexp}
	P_t F = \sum_{n=0}^\infty e^{-nt} I_n(f_n).
\end{equation}

Finally, the following covariance representation (see \cite[Theorem 20.2]{LastPenrose}) will also be useful for us: for $F, G \in \operatorname{dom} D$, 
\begin{align}\label{eq:Mehler}
	\operatorname{Cov}(F(\eta), G(\eta)) &= \int^\infty_0 e^{-t} \int_\X \mathbb{E}[D_xF(\eta) P_t D_x G(\eta)]\,\lambda(\dd x)\dd t \nonumber\\
	&= \int^\infty_0 e^{-t} \int_\X \mathbb{E}[D_xF(\eta) D_x G(\eta^t)]\,\lambda(\dd x)\dd t.
\end{align}

\subsection{Proof of results in Section \ref{sec:chaosresults}}\label{sec:chaosproof}
\begin{proof}[Proof of \autoref{thm:eqv}]
	For the first implication, let $F_s$ be $\eps_s$-superconcentrated. If $F \notin \operatorname{dom} D$, then one trivially has the chaos. Assume now $F \in  \operatorname{dom} D$. Combined with \eqref{eq:Mehler}, we have
	\begin{align*}
		\eps_s  \;\E \int_\X (D_x F_s(\eta_s))^2 \lambda_s(\dd x) \ge \operatorname{Var}(F_s(\eta_s))  = \int^\infty_0 e^{-t} \int_\X \mathbb{E}[D_xF(\eta_s) D_x F(\eta_s^t)]\lambda_s(\dd x)\dd t.	
	\end{align*}
	Now, let $F_s$ have the Fock space representation $F_s =\sum_{n=0}^\infty I_n(f_n)$. Then combining \eqref{eq:DPcomm}, \eqref{eq:Dexp}, \eqref{eq:Ptexp} and \eqref{eq:Iorth}, one has
	\begin{align}\label{eq:Et}
		E_t:&=e^{-t}\int_\X \mathbb{E}[D_xF(\eta_s) D_x F(\eta_s^t)]\lambda_s(\dd x) = \int_\X \mathbb{E}[D_xF(\eta_s) D_x P_t F(\eta_s)]\lambda_s(\dd x) \nonumber\\
		&\qquad =\int_\X \mathbb{E}\Big[\sum_{n=1}^\infty n I_{n-1}(f_n(x,\cdot)) \sum_{n=1}^\infty n e^{-nt}  I_{n-1}(f_n(x,\cdot))\Big]\lambda_s(\dd x)= \sum_{n=1}^\infty n n! e^{-nt} \|f_{n}\|_n^2
	\end{align}
	which is non-negative and non-increasing in $t$. Thus, for any $\delta_s>0$, 
	$$
	\eps_s  \;\E \int_\X (D_x F_s(\eta_s))^2 \lambda_s(\dd x) \ge  \int^{\delta_s}_0E_t \dd t \ge \delta_s E_{\delta_s}.
	$$
	Hence,
	$$
	e^{\delta_s} E_{\delta_s} \le \eps_s \frac{e^{\delta_s}}{\delta_s} \;\E \int_\X (D_x F_s(\eta_s))^2 \lambda_s(\dd x) .
	$$
	Noting from \eqref{eq:Et} that $e^t E_t$ is non-increasing in $t$, we obtain for all $t \ge \delta_s$ that
	$$
	\int_\X \mathbb{E}[D_xF_s(\eta_s) D_x F_s(\eta_s^t)]\lambda_s(\dd x) = e^t E_t \le e^{\delta_s} E_{\delta_s} \le \eps_s \frac{e^{\delta_s}}{\delta_s} \;\E \int (D_x F_s(\eta_s))^2 \lambda_s(\dd x),
	$$
	which by \autoref{definition:chaos} proves that $F_s$ is $(\eps_s e^{\delta_s}/\delta_s, \delta_s)$-chaotic, yielding the first assertion.
	\medskip
	
	Conversely, if $F_s$ is $(\eps_s, \delta_s)$-chaotic, again using that $E_t$ is non-increasing, we have
	\begin{align*}
		\operatorname{Var}(F_s(\eta_s))  &=  \int^{\delta_s}_0 e^{-t} \int_\X\mathbb{E}[D_xF_s(\eta_s) D_x F_s(\eta_s^t)]\lambda_s(\dd x)\dd t + \int_{\delta_s}^\infty e^{-t} \int_\X\mathbb{E}[D_xF_s(\eta_s) D_x F_s(\eta_s^t)]\lambda_s(\dd x)\dd t\\
		& \le \delta_s E_0 + \eps_s \E \int_\X (D_x F_s(\eta_s))^2 \lambda_s(\dd x)\\
		& = (\delta_s + \eps_s) \E \int_\X (D_x F_s(\eta_s))^2 \lambda_s(\dd x),
	\end{align*}
	implying $F_s$ is $(\eps_s+\delta_s)$-superconcentrated, concluding the proof.
\end{proof}

Before proving \autoref{thm:var}, we first prove the following identity.
\begin{lemma}\label{lem:intermed}
	For $\eta$ a Poisson point process on a measurable space $(\mathbb{X},\mathcal{X})$ with $\sigma$-finite intensity measure $\lambda$, and $F\in L^2_\eta$ with $DF \in L^2 (\lambda \otimes \P_\eta)$,
	 \begin{equation}
		e^{-t}  \int_\X \E  \left[D_x F(\eta) D_x F(\eta^t)\right] \lambda(\dd x) = \E \Big[\sum_{x \in \eta \cap \eta^t} D_x^- F(\eta) D_x^- F(\eta^t)\Big].
	\end{equation}
\end{lemma}
\begin{proof}
	      Note that
	\begin{align*}
		\E \Big[\sum_{x \in \eta \cap \eta^t} D_x^- F(\eta) D_x^- F(\eta^t)\Big]
		&= \E \Big[\sum_{x \in \eta} \mathds{1}(x \in \eta^t) D_x^- F(\eta) D_x^- F(\eta^t) \Big]\\
		&=\E \Big[\sum_{x \in \eta} D_x^- F(\eta) \E \left(\mathds{1}(x \in \eta^t) D_x^- F(\eta^t) \big| \eta\right)\Big].
	\end{align*}
	Thus using the Mecke formula \eqref{eq:Mecke}, we obtain
	\begin{align*}
		&\E \Big[\sum_{x \in \eta \cap \eta^t} D_x^- F(\eta) D_x^- F(\eta^t)\Big] \\
		&= \int_\X \E  \left[D_x F(\eta) \E \Big(\mathds{1}(x \in (\eta+\delta_x)^t) [F((\eta+\delta_x)^t) - F((\eta+\delta_x)^t-\delta_x)]  \big| \eta\Big) \right] \lambda(\dd x)\\
		&=\int_\X \E  \left[D_x F(\eta) \E \Big(\mathds{1}(x \in (\eta+\delta_x)^t) [F(\eta^t+\delta_x) - F(\eta^t)]  \big| \eta\Big) \right] \lambda(\dd x).
	\end{align*}
	By the definition of the process $(\eta+\delta_x)^t$, we have $\P(x\in(\eta+\delta_x)^t | \eta)=e^{-t}$. Thus, by the conditional independence of $\mathds{1}(x \in (\eta+\delta_x)^t)$ and $F(\eta^t+\delta_x) - F(\eta^t)$, given $\eta$, we have
	\begin{align*}
		\E \bigg[\sum_{x \in \eta \cap \eta^t} D_x^- F(\eta) D_x^- F(\eta^t)\bigg]
		&=\int_\X e^{-t} \E  \left[D_x F(\eta) \E \Big(F(\eta^t+\delta_x) - F(\eta^t)  \big| \eta\Big) \right] \lambda(\dd x)\\
		&=e^{-t} \int_\X  \E  \left[D_x F(\eta) D_x F(\eta^t)\right] \lambda(\dd x). \qedhere
	\end{align*}
\end{proof}
\begin{proof}[Proof of \autoref{thm:var}]
        The result follows immediately from \autoref{lem:intermed} and \eqref{eq:Mehler}.
    \end{proof}

\begin{proof}[Proof of \autoref{cor:var}]
	Let $D_xF_s(\mu) D_xF_s(\mu')\in \{0,1\}$ for $x \in \X$ and $\mu, \mu' \in \mathbf{N}$, or equivalently $D_x^-F_s(\mu) D_x^-F_s(\mu')\in \{0,1\}$ for $x \in \mu \cap \mu'$ and $\mu, \mu' \in \mathbf{N}$. Thus, from the definition of the random sets $A_s^0$ and $A_s^t$, we obtain
	\begin{align*}
		\int^\infty_0\mathbb{E}\bigg[\sum_{x\in\eta_s\cap\eta_s^t}D_x^-F_s(\eta_s)D_x^-F_s(\eta_s^t)\bigg]\dd t &= \int^\infty_0\mathbb{E}\bigg[\sum_{x\in\eta_s\cap\eta_s^t} \mathds{1}\{D_x^-F_s(\eta_s) \neq 0, D_x^-F_s(\eta_s^t) \neq 0\}\bigg]\dd t\\
		&=\int^\infty_0\mathbb{E}|A_s^0\cap A_s^t|\dd t,
	\end{align*}
	so that the first assertion follows by \autoref{thm:var}.
	\medskip
	
	Now let $F_s$ be $(\eps_s,\delta_s)$-chaotic, that is for all $t \ge \delta_s$,
	\begin{equation}\label{eq:setchaos}
		\int_\X \E \left[D_x F_s(\eta_s) D_x F_s(\eta_s^t)\right] \lambda_s(\dd x)
		\le \eps_s \; \E \int_\X (D_x F_s(\eta_s))^2  \lambda_s(\dd x).
	\end{equation}
	Now by \autoref{lem:intermed}, we have for $t \ge \delta_s$,
	$$
	\int_\X \E \left[D_x F_s(\eta_s) D_x F_s(\eta_s^t)\right] \lambda_s(\dd x) = e^t \mathbb{E}\bigg[\sum_{x\in\eta\cap\eta^t}D_x^-F_s(\eta)D_x^-F_s(\eta^t)\bigg] = e^t \E|A_s^0 \cap A_s^t|,
	$$
	while the Mecke formula \eqref{eq:Mecke} and noting from our assumption that $(DF)^2 = \mathds{1}\{DF \neq 0\} \in \{0,1\}$ yields 
	$$
	\E \int_\X (D_x F_s(\eta_s))^2  \lambda_s(\dd x) = \E\bigg[\sum_{x \in \eta_s} \mathds{1}\{D_x^- F_s(\eta_s) \neq 0\}\bigg] = \E|A_s^0| \equiv \E|A_s|.
	$$
	Thus, from \eqref{eq:setchaos} we obtain that for $t \ge \delta_s$,
	$$
	e^t \E|A_s^0 \cap A_s^t| \le  \eps_s \E|A_s|,
	$$
	showing that $A_s$ is $(\eps_s,\delta_s)$-chaotic. Following the argument in reverse order yields the converse.
\end{proof}

\begin{proof}[Proof of \autoref{cor:oneway}]
	Let $D_x F_s(\mu) D_xF_s(\mu') \in \{0\} \cup[c,\infty)$. Arguing similarly as in the proof of  \autoref{cor:var}, we obtain
	\begin{align*}
		\int^\infty_0\mathbb{E}\bigg[\sum_{x\in\eta_s\cap\eta_s^t}D_x^-F_s(\eta_s)D_x^-F_s(\eta_s^t)\bigg]\dd t &\ge c^2 \int^\infty_0\mathbb{E}\bigg[\sum_{x\in\eta_s\cap\eta_s^t} \mathds{1}\{D_x^-F_s(\eta_s) \neq 0, D_x^-F_s(\eta_s^t) \neq 0\}\bigg]\dd t\\
		&=c^2\int^\infty_0\mathbb{E}|A_s^0\cap A_s^t|\dd t
	\end{align*}
	yielding the first claim via \eqref{eq:Mehler}.
	\medskip
	
	Now let $F_s$ be $(\eps_s,\delta_s)$-chaotic. Again,  arguing  as in the proof of \autoref{cor:var}, we obtain
	$$
	\int_\X \E \left[D_x F_s(\eta_s) D_x F_s(\eta_s^t)\right] \lambda_s(\dd x) \ge c^2 e^t \E|A_s^0 \cap A_s^t|,
	$$
	yielding further that for $t \ge \delta_s$,
	$$
	e^t \E|A_s^0 \cap A_s^t| \le  c^{-2}\eps_s \E|A_s|.
	$$
	This shows that $A_s$ is $(c^{-2}\eps_s,\delta_s)$-chaotic.
\end{proof}
    
\subsection{Proofs of Results in Section \ref{Sec:VarEst}}\label{sec:varproof}
\begin{proof}[Proof of \autoref{thm:L1-L2}]
    Using the definition of $P_t$ in \eqref{eq:Meh} and the Cauchy-Schwarz inequality along with \eqref{eq:contractive} with $p=2$, we have
	\begin{align*}
		&\int^\infty_0 e^{-t}\int_\X \mathbb{E}[g(x,\eta) g(x,\eta^t)]\,\lambda(\dd x)\dd t\\
		&=\int^\infty_0 e^{-t}\int_\X \mathbb{E}[g(x,\eta) P_t g(x,\eta)]\,\lambda(\dd x)\dd t\\
		&\le \int^\infty_0 e^{-t}\int_\X \|g(x,\eta)\|_{L_\eta^2} \|P_t g(x,\eta)\|_{L_\eta^2} \,\lambda(\dd x)\dd t\\
		& \le  \int^\infty_0 e^{-t}\int_\X \|g(x,\eta)\|_{L_\eta^2} \|g(x,\eta)\|_{L_\eta^2} \,\lambda(\dd x)\dd t
	\end{align*}
	yielding the first assertion. 
	\medskip
	
	For the second assertion, we argue as in the proof of \cite[Theorem 1.6]{NourdinPeccatiYang}. Instead of Cauchy-Schwarz inequality, we now use H\"{o}lder's inequality to obtain
	\begin{equation*}
		\int^\infty_0 e^{-t}\int_\X \mathbb{E}[g(x,\eta) P_t g(x,\eta)]\,\lambda(\dd x)\dd t
		\le \int^\infty_0 e^{-t}\int_\X \|g(x,\eta)\|_{L_\eta^{1+e^{-t}}} \|P_t g(x,\eta)\|_{L_\eta^{1+e^{t}}} \,\lambda(\dd x)\dd t.
	\end{equation*}
	Since $D_y g(x,\eta) \le 0$ for $\lambda$-almost every $x,y \in \X$ by our assumption, using the restricted hypercontractivity \eqref{eq:resthyp} of the semigroup $(P_t)_{t \ge 0}$ with $p=2$ in the first step, changing variable $v=1+e^{-t}$ in the second, and using the H\"{o}lder inequality again in the final step, we obtain
	\begin{multline*}
		\int^\infty_0 e^{-t} \int_\X \mathbb{E}[g(x,\eta) g(x,\eta^t)]\,\lambda(\dd x)\dd t \le  \int^\infty_0 e^{-t} \int_\X\|g(x,\eta)\|_{L_\eta^{1+e^{-t}}} \|g(x,\eta)\|_{L_\eta^2} \lambda(\dd x)\dd t\\
		 = \int^2_1 \int_\X\|g(x,\eta)\|_{L_\eta^v} \|g(x,\eta)\|_{L_\eta^2}  \lambda(\dd x)\dd v
		\le  \int^2_1 \int_\X \|g(x,\eta)\|_{L_\eta^1}^{\frac{2-v}{v}} \|g(x,\eta)\|_{L_\eta^2}^{\frac{2(v-1)}{v}+1}  \lambda(\dd x)\dd v.
	\end{multline*}
	Setting $b= \|g(x,\eta)\|_{L_\eta^1}/ \|g(x,\eta)\|_{L_\eta^2}$ then yields
	\begin{multline*}
		\int^\infty_0 e^{-t} \int_\X\mathbb{E}[g(x,\eta) g(x,\eta^t)]\,\lambda(\dd x)\dd t \le  \int^2_1 \int_\X b^{\frac{2}{v}-1}
		\|g(x,\eta)\|_{L_\eta^2}^{2}  \lambda(\dd x)\dd v\\
		\le 2 \int^1_0 \int_\X b^{u} \|g(x,\eta)\|_{L_\eta^2}^{2}  \lambda(\dd x)\dd u \le  2 \int_\X \frac{\|g(x,\eta)\|_{L_\eta^2}^{2}}{1+(1/2) \log (1/b)}  \lambda(\dd x)
	\end{multline*}
	concluding the proof.
\end{proof}

Next, we proceed with proving  \autoref{thm:Varlb}. Since $g(x):= g(x,\cdot) \in L^2_\eta$ for $\lambda$-almost every $x$, by \eqref{eq:FS} it admits the Fock space representation given by
\begin{equation}
	g(x)(\eta) = g(x,\eta) = \mathbb{E} [g(x,\eta)] + \sum_{n=1}^\infty I_n(f_n^x) = \sum_{n=0}^\infty  I_n(f_n^x).
\end{equation}
Note also that by the orthogonality \eqref{eq:Iorth} of the Wiener-Ito integrals, one further has
$$
\|g(x,\eta)\|_{L^2_\eta}^2=\mathbb{E} [g(x,\eta)^2 ]=   \sum_{n=0}^\infty n! \|f_n^x\|_n^2.
$$

\begin{proof}[Proof of \autoref{thm:Varlb}]
			We will argue similarly as in the proof of \cite[Theorem 1.1]{SchulteTrapp}. We have
			\begin{align*}
				\mathbb{E} \left[\int_\X g(x,\eta)^2\, \lambda(\dd x)\right] =  \int_\X  \sum_{n=0}^\infty n! \|f_n^x\|_n^2 \, \lambda(\dd x) = \sum_{n=0}^\infty n! w_n,
			\end{align*}
			where for $n \ge 0$, we write $w_n = \int_\X \|f_n^x\|_n^2 \,\lambda(\dd x) \in [0,\infty)$ (since $g \in  L^2 (\lambda \otimes \P_\eta)$). On the other hand, by our assumption, $g(x,\eta) \in \operatorname{dom} D$ for $\lambda$-a.e.\ $x\in \X$. Thus, arguing as we do above considering the decomposition \eqref{eq:Dexp} for $D_y g(x,\eta)$, we obtain
			\begin{multline*}
				\mathbb{E} \left[\int_{\X^2} \left(D_y g(x,\eta) \right)^2\, \lambda^2(\dd x,\dd y)\right] =\int_{\X^2}  \sum_{n=0}^\infty n! \left\|\frac{1}{n!} \mathbb{E} D^{n}_{\cdot, \cdot, \hdots, \cdot} D_y f(x,\eta) \right\|_n^2 \lambda^2(\dd x,\dd y)\\
				= \sum_{n=0}^\infty (n+1)^2 n! \int_\X \|f_{n+1}^x\|_{n+1}^2 \lambda(\dd x)=  \sum_{n=1}^\infty n n! \int_\X \|f_{n}^x\|_n^2 \lambda(\dd x) =  \sum_{n=0}^\infty n n! w_n.
			\end{multline*}
			Thus, our assumption translates to the condition that 
			$$
			\sum_{n=0}^\infty n! w_n (\alpha - n) \ge 0.
			$$
			Next, employing a similar argument, by \eqref{eq:Ptexp}, we also have
			\begin{align}\label{eq:varianceexp}
				&\int^\infty_0 e^{-t} \int_\X \mathbb{E}[g(x,\eta) g(x,\eta^t)]\,\lambda(\dd x)\dd t \nonumber\\
				&= \int^\infty_0 e^{-t}\int_\X \mathbb{E}\left[g(x,\eta) P_t g(x,\eta) \right]\,\lambda(\dd x)\dd t \nonumber\\
				&=\int^\infty_0 e^{-t}\int_\X \mathbb{E}\bigg[\Big(\sum_{n=0}^\infty  I_n(f_n^x) \Big) \Big(\sum_{n=0}^\infty  e^{-nt}I_n(f_n^x) \Big) \bigg]\,\lambda(\dd x)\dd t \nonumber\\
				&= \int^\infty_0 \int_\X \sum_{n=0}^\infty  e^{-(n+1)t} n! \|f_n^x\|_n^2 \,\lambda(\dd x)\dd t =  \sum_{n=0}^\infty \frac{n!}{n+1} w_n.
			\end{align}
			Thus,
			\begin{multline*}
				(1+\alpha)^2	\int^\infty_0 e^{-t} \int_\X \mathbb{E}[g(x,\eta) g(x,\eta^t)]\,\lambda(\dd x)\dd t - (1+\alpha) \mathbb{E} \left[\int_\X g(x,\eta)^2 \,\lambda(\dd x)\right]\\
				= \sum_{n=0}^\infty n! w_n \left(\frac{(1+\alpha)^2}{n+1} - (1+\alpha)\right) \ge  \sum_{n=0}^\infty n! w_n \left(\alpha - n \right) \ge 0, 
			\end{multline*}
			where for the penultimate step, we use the fact that for all $n \ge 0$, we have $n! w_n \ge 0$ and
			$$
			\frac{(1+\alpha)^2}{n+1}  \ge (2\alpha - n+1),
			$$
			which is equivalent to the trivial inequality
			$(n-\alpha)^2 \ge 0$. This yields the result.
	\end{proof}

\begin{proof}[Proof of \autoref{lem:T_3order}]
	Since $F \in \operatorname{dom} D$ and $g_1, g_2$ have disjoint supports, $g_1(x) = g_1(x,\cdot)$ and $g_2(x) = g_2(x,\cdot)$ are square integrable for $\lambda$-a.e.\ $x \in \X$. Let their Fock space representations be given by
	$$
	g_1(x,\eta) =  \sum_{n=0}^\infty  I_n(f_{n,1}^x) \quad \text{and} \quad g_2(x,\eta) =  \sum_{n=0}^\infty  I_n(f_{n,2}^x).
	$$
Denote for $i \in \{1,2\}$ the functions
	$$
	f_{n+1,i}(x,x_1, \hdots,x_n) = f_{n,i}^x(x_1, \hdots,x_n).
	$$
	Arguing as for \eqref{eq:varianceexp}, we obtain for $i \in \{1,2\}$ that
	\begin{align*}
			&\int^\infty_0 e^{-t} \int_\X\mathbb{E}[g_i(x,\eta) g_i(x,\eta^t)]\, \lambda(\dd x)\dd t \\
			&= \int^\infty_0 \int_\X \sum_{n=0}^\infty  e^{-(n+1)t} n! \|f_{n,i}^x\|_n^2 \, \lambda(\dd x)\dd t = \sum_{n=0}^\infty \frac{n!}{n+1} \|f_{n+1,i}\|_{n+1}^2
	\end{align*}
	so that by our assumption, we have
	\begin{equation}\label{eq:varianceorder'}
		\sum_{n=0}^\infty \frac{n!}{n+1} \|f_{n+1,1}\|_{n+1}^2 \gg \sum_{n=0}^\infty \frac{n!}{n+1} \|f_{n+1,2}\|_{n+1}^2.
	\end{equation}
Similarly,
\begin{align*}
		&\int^\infty_0 e^{-t} \int_\X \mathbb{E}[g_1(x,\eta) g_2(x,\eta^t)] \, \lambda(\dd x)\dd t 
		=  \sum_{n=0}^\infty \frac{n!}{n+1} \langle f_{n+1,1}, f_{n+1,2}\rangle_{n+1}.
\end{align*}
	Let $p_n \in (0,1)$, $n \ge 0$ be a sequence such that $p_n n!/(n+1) \in (0,1)$ and $\sum_{n=0}^\infty p_n n!/(n+1) = 1$, which gives rise to a probability distribution. Let $N \in \mathbb{N}_0$ be a random variable such that $\P(N=n) = p_n n!/(n+1)$, $n \ge 0$. Then, using the Cauchy-Schwarz inequality twice, we have
	\begin{multline*}
	\sum_{n=0}^\infty \frac{n!}{n+1} \langle f_{n+1,1}, f_{n+1,2}\rangle_{n+1} = \sum_{n=0}^\infty p_n \frac{n!}{n+1} \langle p_n^{-1/2}f_{n+1,1}, p_n^{-1/2}f_{n+1,2}\rangle_{n+1}\\
	= \mathbb{E} \langle p_N^{-1/2}f_{N+1,1}, p_N^{-1/2}f_{N+1,2}\rangle_{N+1} \le \mathbb{E} \|p_N^{-1/2}f_{N+1,1}\|_{N+1} \|p_N^{-1/2}f_{N+1,2}\|_{N+1}\\
	\le \sqrt{\mathbb{E}  \|p_N^{-1/2}f_{N+1,1}\|_{N+1}^2 \, \mathbb{E}  \|p_N^{-1/2}f_{N+1,2}\|_{N+1}^2}\\
	= \sqrt{\sum_{n=0}^\infty \frac{n!}{n+1} \|f_{n+1,1}\|_{n+1}^2 \, \sum_{n=0}^\infty \frac{n!}{n+1} \|f_{n+1,2}\|_{n+1}^2} \ll \sum_{n=0}^\infty \frac{n!}{n+1} \|f_{n+1,1}\|_{n+1}^2,
	\end{multline*}
	where the final step is due to \eqref{eq:varianceorder'}. This concludes the proof.
\end{proof}

\begin{proof}[Proofs of Corollaries \ref{cor:bddD1}, \ref{cor:bddD2} and \ref{cor:T3}]
	 \autoref{cor:bddD1} is immediate from \autoref{thm:L1-L2} upon noting via an argument similar to the proof of \autoref{lem:intermed} that
	$$
	e^{-t} \int_\X \mathbb{E}[g(x,\eta) g(x,\eta^t)]\lambda(\dd x) = \E |P^0 \cap P^t|.
	$$
	Similar equalities yield Corollaries \ref{cor:bddD2} and \ref{cor:T3} from \autoref{thm:Varlb} and \autoref{lem:T_3order}, respectively.
\end{proof}

\subsection{Proof of \autoref{thm:supcon_sos}}\label{sec:supproof}
Observe that
\begin{align*}
	D_yF_s(\eta_s)=\sum\limits_{x\in\eta_s+\delta_y} f_s(x,\eta_s+\delta_y)-\sum\limits_{x\in\eta_s} f_s(x,\eta_s)=f_s(y,\eta_s+\delta_y)+\sum\limits_{x\in\eta_s}D_yf_s(x,\eta_s).
\end{align*}
Thus, we may write $D_yF_s(\eta_s)=g_1(y,\eta_s)-g_2(y,\eta_s)$, where we let $g_1(y,\eta_s):=f_s(y,\eta_s+\delta_y)$ and $g_2(y,\eta_s):=-\sum\limits_{x\in\eta_s}D_yf_s(x,\eta_s)$.
The Poincar\'e inequality (\ref{thm:poincare}) then yields
\begin{align*}
	\Var(F_s(\eta))\leq\E \int_\X (D_x F_s(\eta_s))^2 \, \lambda_s(\dd x)
	= \int _\X\E [g_1(x,\eta_s)^2+g_2(x,\eta_s)^2-2g_1(x,\eta_s)g_2(x,\eta_s)] \,\lambda_s(\dd x).
\end{align*}
Now by assumption \ref{A2}, we have that $\E g_1(x,\eta_s)^2  \le \eps_s \E g_2(x,\eta_s)^2$ for $\lambda_s$-a.e.\ $x \in \X$. Further, note by the Cauchy-Schwarz inequality and assumption \ref{A2} that for $\lambda_s$-a.e.\ $x \in \X$,
$$
\E[g_1(x,\eta_s)g_2(x,\eta_s)]\leq\sqrt{{\E[g_1(x,\eta_s)^2]\E[g_2(x,\eta_s)^2]}}\le \sqrt{\eps_s} \E[g_2(x,\eta_s)^2].
$$
Combining this with the variance bound above yields
\begin{equation}\label{eq:Pub}
\E \int_\X (D_x F_s(\eta_s))^2 \, \lambda_s(\dd x)\ge (1-\sqrt{\eps_s})^2 \int_\X \E [g_2(x,\eta_s)^2] \,\lambda_s(\dd x),
\end{equation}
that is, the upper bound from the Poincar\'e inequality is exactly of the order $\int \E [g_2(x,\eta_s)^2] \lambda_s(\dd x)$. Now we show a tighter bound for the variance by a more careful variance estimation. By the Mecke formula \eqref{eq:Mecke}, we have
\begin{align*}
	\Var&(F_s)=\E F_s^2-(\E F_s)^2 \\
	&=\E\bigg[\sum\limits_{x\in \eta}f_s^2(x,\eta_s)+\sum\limits_{\substack{x_1,x_2\in\eta_s \\ x_1\neq x_2}}f_s(x_1,\eta)f_s(x_2,\eta_s)\bigg]-\left[\int_\X \E[f_s(x,\eta_s+\delta_x)]\, \lambda_s(\dd x)\right]^2 \\
	&=\int_\X \E[f_s^2(x,\eta_s+\delta_x)]\, \lambda_s(\dd x)+\int_{\X^2}\E[f_s(x_1,\eta_s+\delta_{x_1}+\delta_{x_2})f_s(x_2,\eta_s+\delta_{x_1}+\delta_{x_2})]\, \lambda_s^2(\dd x_1,\dd x_2) \\ &\qquad\qquad\qquad\qquad\qquad\qquad-\int_{\X^2}\E[f_s(x_1,\eta_s+\delta_{x_1})]\E[f_s(x_2,\eta_s+\delta_{x_2})] \,\lambda_s^2(\dd x_1,\dd x_2) \\
	&= \int_\X \E[g_1(x,\eta_s)^2]\, \lambda_s (\dd x)+T_s \le (1+C) \int_\X \E[g_1(x,\eta_s)^2]\, \lambda_s (\dd x),
\end{align*}
where the final step is by assumption \ref{A1} with $T_s$ as therein. This yields the first assertion. 

\medskip

Finally, assumption \ref{A2} implies
$$
\Var(F_s) \le (1+C) \eps_s \int_\X \E[g_2(x,\eta_s)^2]\, \lambda_s (\dd x) \le \frac{(1+C) \eps_s}{(1-\sqrt{\eps_s})^2} \E \int_\X (D_x F_s(\eta_s))^2 \, \lambda_s(\dd x),
$$
where the final step is due to \eqref{eq:Pub}. This concludes the proof.
\qed

\section*{Acknowledgements}
\addcontentsline{toc}{section}{Acknowledgements}

Both authors were supported by the German Research Foundation (DFG) Project 531540467. We thank Matthias Schulte for bringing this question to our attention, and both him and Giovanni Peccati for many helpful discussions.

\bibliographystyle{abbrv}
\bibliography{refs}
\end{document}